\documentclass[11pt,reqno]{amsart}
\usepackage{amsmath}
\usepackage{amssymb}
\usepackage{tikz}
\usepackage{color}
\usepackage{xcolor}
\usepackage{float}
\usepackage{comment}
\usepackage{geometry}

 \theoremstyle{definition}

 \newtheorem{theorem}{Theorem}[section]%
 \newtheorem{corollary}[theorem]{Corollary}
 \newtheorem{lemma}[theorem]{Lemma}

 \newtheorem{definition}[theorem]{Definition}
 
  \numberwithin{equation}{section}
\newtheorem*{mrst}{Main Result}
\newtheorem*{example}{Example}

\theoremstyle{remark}

\newcommand\R{\mathbb R}
\newcommand\N{\mathbb N}

\newcommand{\id}{\mathrm{id}}

\newcommand\Wp{\mathcal{W}_p}
\newcommand\Wt{\mathcal{W}_2}
\newcommand\Wo{\mathcal{W}_1}
\newcommand\dwpp{d_{\Wp}^p}
\newcommand\dwp{d_{\Wp}}
\newcommand\dwt{d_{\Wt}}
\newcommand{\dwo}{d_{\mathcal{W}_1}}

\newcommand{\isom}{\mathrm{Isom}}

\newcommand{\prob}{\mathcal{P}}

\newcommand{\tgamma}{\widetilde{\gamma}}

\newcommand*{\bigchi}{\mbox{\Large$\chi$}}

\newcommand{\ler}[1]{\left( #1 \right)}
\newcommand{\lesq}[1]{\left[ #1 \right]}

\newcommand{\dd}{\,\mathrm{d}}
\newcommand{\dt}{\mathrm{d}t}
\newcommand{\dx}{\mathrm{d}x}

\newcommand{\x}{x}
\newcommand{\h}{h}
\newcommand{\0}{0}
\newcommand{\inner}[2]{\left< #1,#2 \right>}

\newcommand{\norm}[1]{\left|\left|#1\right|\right|}

\newcommand{\potmu}{\mathcal{T}_\mu^p}
\newcommand{\potmuf}{\mathcal{T}_\mu^4}

\newcommand{\potphimu}{\mathcal{T}_{\Phi(\mu)}^p}
\newcommand{\func}{\mathcal{G}}
\newcommand{\supp}{\mathrm{supp}}

\newcommand{\proj}{\mathfrak{p}}

\newcommand{\ppi}{{\boldsymbol\pi}}

\newcommand{\neighb}{\overset{n}{\sim}}

\title[The isometry group of Wasserstein spaces: the Hilbertian case
]{The isometry group of Wasserstein spaces:\\ the Hilbertian case}                                                                                            

\author[Gy\"orgy P\'al Geh\'er]{Gy\"orgy P\'al Geh\'er}
\address{Gy\"orgy P\'al Geh\'er, Department of Mathematics and Statistics\\ University of Reading\\ Whiteknights\\ P.O.
Box 220\\ Reading RG6 6AX\\ United Kingdom}
\email{gehergyuri@gmail.com}

\author[Tam\'as Titkos]{Tam\'as Titkos}
\address{Tam\'as Titkos, Alfr\'ed R\'enyi Institute of Mathematics\\ Hungarian Academy of Sciences\\ Re\'altanoda u. 13-15.\\
Budapest H-1053\\ Hungary\\ and BBS University of Applied Sciences\\ Alkotm\'any u. 9.\\
Budapest H-1054\\ Hungary}

\email{titkos.tamas@renyi.mta.hu \newline http://renyi.hu/\~{}titkos}

\author[D\'aniel Virosztek]{D\'aniel Virosztek}
\address{D\'aniel Virosztek, Institute of Science and Technology Austria \\ Am Campus 1 \\ 3400 Klos\-ter\-neuburg \\ Austria \\
and Alfr\'ed R\'enyi Institute of Mathematics\\ Hungarian Academy of Sciences\\ Re\'altanoda u. 13-15.\\
Budapest H-1053\\ Hungary}

\email{virosztek.daniel@renyi.hu \newline https://users.renyi.hu/\~{}dviroszt}

\begin{document}


\thanks{Geh\'er was supported by the Leverhulme Trust Early Career Fellowship (ECF-2018-125), and also by the Hungarian National Research, Development and Innovation
Office - NKFIH (grant no. K115383 and K134944).}
\thanks{Titkos was supported by the Hungarian National Research, Development and Innovation Office - NKFIH (grant no. PD128374, grant no. K115383 and K134944), by the János Bolyai Research Scholarship of the Hungarian Academy of Sciences, and by the ÚNKP-20-5-BGE-1 New National Excellence Program of the Ministry of Innovation and Technology.}
\thanks{Virosztek was supported by the European Union’s Horizon 2020 research and innovation program under the Marie Sklodowska-Curie Grant Agreement No. 846294, by the Momentum program of the Hungarian Academy of Sciences under grant agreement no. LP2021-15/2021, and partially supported by the Hungarian National Research, Development and Innovation Office - NKFIH (grants no. K124152 and no. KH129601).}

\begin{abstract}
Motivated by Kloeckner's result on the isometry group of the quadratic Wasserstein space $\mathcal{W}_2(\mathbb{R}^n)$, we describe the isometry group $\isom(\Wp(E))$ for all parameters $0 < p < \infty$ and for all separable real Hilbert spaces $E.$ In particular, we show that $\Wp(X)$ is isometrically rigid for all Polish space $X$ whenever $0<p<1$.  This is a consequence of our more general result: we prove that $\mathcal{W}_1(X)$ is isometrically rigid if $X$ is a complete separable metric space that satisfies the strict triangle inequality. Furthermore, we show that this latter rigidity result does not generalise to parameters $p>1$, by solving Kloeckner's problem affirmatively on the existence of mass-splitting isometries.
\end{abstract}

\maketitle
\tableofcontents

\section{Introduction and main results}\label{s:intro}

Let $(X,\rho)$ be a complete and separable metric space, and denote by $\mathcal{P}(X)$ the set of all Borel probability measures on $X$. Due to many nice geometric features, transport related metrics and techniques on $\mathcal{P}(X)$ have received increased attention in both pure and applied mathematics recently. We mention here only a few papers \cite{Figalli,hairer,navier-stokes,ml2,LV}, for a comprehensive overview and for more references we refer the reader to Ambrosio's, Santambrogio's and Villani's textbooks \cite{Ambrosio,Santambrogio,villani-book,villani-ams-book}.
Probably the most important transport related metric on sufficiently concentrated probability measures is the so-called $p$-Wasserstein metric ($0<p<\infty$). Bertrand and Kloeckner dedicated a whole series of papers \cite{bertrand-kloeckner-hadamard,bertrand-kloeckner-2016,Kloeckner-2010,kloeckner-hausdorff, ultrametric} to understand and describe some important geometric properties of $2$-Wasserstein spaces including the structure of their isometries. For more results concerning the strucutre of isometries with respect to different probability metrics we refer the reader to the papers \cite{dolinar-molnar,kuiper,geher-titkos,isemb-jmaa, molnar-levy,virosztek-asm-szeg}.

We highlight the paper \cite{Kloeckner-2010}, since it serves as the main motivation for our work. In that paper Kloeckner described the isometry group of $\mathcal{W}_2(\R^n)$, the quadratic Wasserstein space built on $\R^n$ (see the precise definition later). When describing the isometry group of a metric space of measures, it is a standard phenomenon that isometries of the underlying structure appear by means of push-forward. These isometries are called trivial isometries. It is a natural question whether all isometries of $\Wt(\R^n)$ are trivial, in other words, whether the isometry group of $\Wt(\R^n)$ is isomorphic to the isometry group of $\R^n$? Kloeckner showed that the answer to this question is negative in the case of $\Wt(\R^n)$, moreover, there is an important difference between the cases when the underlying Euclidean space is one-dimensional, and when it is multi-dimensional. On the one hand, if $n\geq2$ then every isometry $\Phi$ of $\Wt(\R^n)$ has a special feature: they preserve the shape of measures. This means that for all measures $\mu$ there exists an isometry $\psi_{\mu}$ of $\R^n$ (depending on $\mu$) such that $\Phi(\mu)$ is the push-forward of $\mu$ with respect to $\psi_\mu$. On the other hand, the isometry group of $\Wt(\R)$ contains a one-parameter subgroup of wildly behaving elements that do not even preserve the shape of measures. Such isometries are called exotic isometries. Motivated by this latter striking result, in \cite{GTV} we gave a complete characterisation of isometries of $p$-Wasserstein spaces built on the real line $\R$ for all parameters $1\leq p<\infty$. It turned out that the $p=2$ case is exceptional in the sense that if $p\neq2$ then all isometries of $\Wp(\R)$ are trivial.\\

Our aim in this paper is to present a broad extension of Kloeckner's multidimensional results ($\mathcal{W}_2(\R^n)$ $n\geq2)$ on the isometry group, namely 
\begin{itemize}
    \item[-] \emph{to handle the case of arbitrary $0<p<\infty$}
    \item[-] \emph{to drop the assumption of finite-dimensionality.}
\end{itemize}
It turns out that the case $p=2$ is again exceptional: for any separable real Hilbert space $E$ with $\dim E\geq2$ the Wasserstein space $\Wp(E)$ admits nontrivial isometries if and only if $p=2$. The main results of this paper can be informally summarized as follows. 

\begin{mrst}\emph{Let $E$ be a separable real Hilbert space of dimension at least two. For a positive real number $p$ let us denote the $p$-Wasserstein space built on $E$ by $\Wp(E)$. Assume that $\Phi$ is a distance preserving bijection, i.e. an isometry of $\Wp(E)$.}
\begin{itemize}
    \item[(a)] \emph{If $p\neq2$, then $\Phi$ is necessarily a push-forward of an isometry $\psi$ of $E$, that is}
    \begin{equation*}
    \Phi(\mu)={\psi}_{\#}\mu\qquad\big(\mu\in\Wp(E)\big).
    \end{equation*}
    
    \item[(b)]\emph{If $p=2$ and $E$ is infinite dimensional then $\Phi$ can be written as the following composition:}
    \begin{equation*}
    \Phi(\mu) = \ler{\psi \circ t_{m(\mu)}\circ R \circ t_{m(\mu)}^{-1}}_\#\mu \qquad (\mu\in\Wt(E)),
    \end{equation*}
    \emph{where $\psi\colon E\to E$ is an affine isometry, $R\colon E\to E$ is a linear isometry, and $t_{m(\mu)}\colon E\to E$ is the translation on $E$ by the barycenter $m(\mu)$ of $\mu$.}
\end{itemize}
\end{mrst}
Part (b) is a natural extension of Kloeckner's results on $\mathcal{W}_2(\R^n)$, while in part (a) we developed essentially new techniques to prove isometric rigidity. As the value of $p$ affects basic properties of the cost function, the proof of the above statement has to be divided into four separate cases.
In Subsection \ref{ss:recovery} we handle the case when $1\leq p<\infty$ and $p$ is not an even integer. Due to the fact that the cost function is not smooth in one point, as a key step of the proof of Theorem \ref{thm:noteven}, we recover the atoms of any measure $\mu$ by means of the following potential function:
\begin{equation*}
	\potmu\colon E\to \R, \quad x\mapsto \dwpp(\mu,\delta_x) = \int_E \norm{x-y}^p \dd\mu(y)
\end{equation*}
where $\dwp$ is the Wasserstein distance defined in \eqref{eq:wasser_def} below. 
When $p$ is an even positive integer, this potential function does not carry enough information to identify measures. Note that Kloeckner's method to prove the finite dimensional version of (b) above does not work in our infinite dimensional setting, as he uses absolutely continuous measures that have no analogue in infinite dimension. 
We prove (b) in Subsection \ref{ss:2} as Theorem \ref{thm:2}. In Subsection \ref{ss:468} we show that for $p = 2k$ with $k\in\N, k\geq2$ isometries map measures supported on a line into measures supported on another line, which allows us to utilise our recent result from \cite{GTV}, see Theorem \ref{thm:468}. 
Finally, in Section \ref{s:ple1} we prove isometric rigidity of 1-Wasserstein spaces over metric spaces $(X,\rho)$ satisfying the strict triangle inequality 
\begin{equation*}
    \rho(x,y) < \rho(x,z) + \rho(z,y) \quad (x,y,z,\in X, z\notin\{x,y\}),
\end{equation*}
see Theorem \ref{thm:concave-cost}. As a consequence we obtain isometric rigidity of $\Wp(E)$ for the concave case, $0<p<1$.  In fact, our argument shows that $\Wp(X)$ is isometrically rigid for every
Polish space $X$ if $0<p<1$.

The starting point in each of the above cases will be to see that any isometry maps Dirac measures into Dirac measures, that is, they do not split mass. Let us point out that one has to be cautious here. Although isometries do not split mass in the cases that were investigated earlier, Kloeckner posed the following problem in \cite[Question 2]{Kloeckner-2010}:
\begin{center}
    \emph{\lq\lq Does there exist a Polish (or Hadamard) space $X$ whose Wasserstein space $\mathcal{W}_2(X)$ possesses an isometry that does not preserve the set of Dirac measures?''}
\end{center}
We shall see in Section \ref{s:example} that, contradicting to intuition, such a Polish space exists for all parameters $p\geq 1$. In \cite{GTV} we showed that the $1$-Wasserstein space built over the line segment $[0,1]$ possesses isometries that send Dirac measures into measures typically supported on two points -- hence split mass. Using this example, in Section \ref{s:example} we construct another Polish space which illustrates that the answer to the above question is indeed affirmative for all parameters $p \geq 1$. Furthermore, this will also show that the above mentioned Theorem \ref{thm:concave-cost} is sharp in the sense that $\Wp(X)$ spaces over metric spaces satisfying the strict triangle inequality are not isometrically rigid in general if $p>1$.

We note that as mass-splitting isometries are clearly exotic and hence also non-trivial, our construction described in Section \ref{s:example} solves another open problem of Kloeckner affirmatively (\cite[Question 1]{Kloeckner-2010}):

\begin{center}
    \emph{\lq\lq Does there exist a Polish (or Hadamard) space $X \neq \R$ such that $\mathcal{W}_2(X)$ admits exotic isometries? Does there exist a Polish (or Hadamard) space $X \neq \R^n$ such that $\mathcal{W}_2(X)$ admits non-trivial isometries?''}
\end{center}
\smallskip

Now, we set the terminology. Let $\mu,\nu\in\prob(X)$ be two Borel probability measures on the complete and separable metric space $\big(X,\rho\big)$.  The support $\supp(\mu)$ of a $\mu\in\prob(X)$ is defined to be the smallest closed subset of $X$ for which every open neighbourhood of every point of the set has positive measure.
A Borel probability measure $\pi$ on $X \times X$ is said to be a \emph{coupling} of (or \emph{transport plan} for) $\mu$ and $\nu$ if the marginals of $\pi$ are $\mu$ and $\nu$, that is, $\pi\ler{A \times X}=\mu(A)$ and $\pi\ler{X \times B}=\nu(B)$ for all Borel sets $A,B\subseteq X$. The set of all couplings is denoted by $\Pi(\mu,\nu)$. For any parameter value $0<p<\infty$ one can define the \emph{$p$-Wasserstein space} $\Wp(X)$ as the set of all $\mu\in\prob(X)$ that satisfy $\int_X \rho(x,\hat{x})^p~\mathrm{d}\mu(x)<\infty$ for some (hence all) $\hat{x}\in X$, endowed with the \emph{$p$-Wasserstein distance}
\begin{equation} \label{eq:wasser_def}
\dwp\ler{\mu, \nu}:=\ler{\inf_{\pi \in \Pi(\mu, \nu)} \int_{X \times X} \rho(x,y)^p~\dd \pi(x,y)}^{\min\big\{{\frac{1}{p}},1\big\}}.
\end{equation}
A coupling $\pi\in\Pi(\mu,\nu)$ is called an \emph{optimal coupling} if the infimum in \eqref{eq:wasser_def} is a minimum and it is attained at $\pi$. The set of all optimal couplings for $\mu$ and $\nu$ is denoted by $\Pi^0(\mu,\nu)$. 

Distance preserving bijections are termed as $\emph{isometries}$ and the symbol $\isom(\cdot)$ refers to the isometry group. The \emph{push-forward map} $g_\# \colon \prob(X)\to\prob(X)$ induced by a measurable function $g\colon X \rightarrow X$ is defined by $\big(g_\#(\mu)\big)(A)=\mu(g^{-1}[A])$
for all $A\subseteq X$ Borel set and $\mu\in\prob(X)$, where $g^{-1}[A]=\{x\in X\,\colon\, g(x)\in A\}$.

The set of all Dirac measures is denoted by $\Delta(X)$. Note that if $1\leq p<\infty$, then $\Wp(X)$ contains an isometric copy of $X$, since the embedding
\begin{equation*}
\iota\colon X\to\Wp(X),\qquad \iota(x)=\delta_x
\end{equation*}
is distance preserving. Moreover, the image of $\iota$ can be considered as the core of the Wasserstein space in the sense that its convex span -- the set of finitely supported probability measures -- is a dense subset of $\Wp(X)$ with respect to the topology of weak convergence.

For a given $\psi\in\isom(X)$ the push-forward map $\psi_\#$ belongs to $\isom(\Wp(X))$ for all $0<p<\infty$, and the action of $\psi_\#$ on $\Delta(X)$ is given by $\psi_\#(\delta_x)=\delta_{\psi(x)}$. Isometries of this push-forward type are termed as \emph{trivial isometries}. A $p$-Wasserstein space $\Wp(X)$ is called \emph{isometrically rigid} if the push-forward map
$$\#\colon\isom(X)\to\isom(\Wp(X));\qquad \psi\mapsto \psi_{\#}$$
is surjective, in other words, if every isometry is trivial. Let us remark that if an isometry $\Phi\in\isom(\Wp(X))$ maps $\Delta(X)$ onto $\Delta(X)$, then it defines a map $\psi:X\to X$ via the identity 
\begin{equation}\label{E:hatas}
\Phi(\delta_x)=\delta_{\psi(x)}\qquad(x\in X).
\end{equation}
The map $\psi$ clearly belongs to $\isom(X)$, since $\dwp(\delta_x,\delta_y)=\rho(x,y)^{\min\{p,1\}}$ for all $x,y\in X$. Nonetheless, we have to be careful for two reasons:
\begin{itemize}
    \item[-] in general nothing guarantees that an isometry maps $\Delta(X)$ onto itself, that is, a Wasserstein space may possess mass-splitting isometries, see Section \ref{s:example}; and
    \item[-] even if it does, \eqref{E:hatas} does not imply the form $\Phi=\psi_{\#}$, cf. part (b) of the Main Theorem and \cite[Theorems 1.1 and 1.2]{Kloeckner-2010}.
\end{itemize}

\section{Kloeckner's problems on mass splitting isometries}\label{s:example}
In this section we answer the aforementioned questions \cite[Questions 1-2]{Kloeckner-2010} affirmatively by showing that for all $p > 1$ there exists a Polish space $X$ such that the $p$-Wasserstein space $\Wp(X)$ possesses mass splitting isometries. 
For $p=1$ this question was recently answered by the authors in \cite{GTV}. We recall the details below, as we shall manipulate this example in order to answer the question in the case of strictly convex cost ($p>1$).
Let us denote by $Y$ the complete separable metric space $([0,1],|\cdot|)$. A special feature of this space is that the $1$-Wasserstein distance in $\Wo(Y)$ can be calculated by means of cumulative distribution functions and quantile functions. With elementary manipulations, both functions can be considered as right-continuous $[0,1]\to[0,1]$-type functions. The cumulative distribution function can be defined as 
\begin{equation*}F_\mu(x):=\mu([0,x])\qquad(x\in[0,1]),
\end{equation*}
while the \emph{quantile function} of $\mu$ is is defined by
\begin{equation*}
 F_\mu^{-1}(y):=\sup \left\{ x\in\mathbb{R} \, : \, F_\mu(x) \leq y \right\}\qquad(x\in(0,1)).\end{equation*}
In order to obtain a $[0,1]\to[0,1]$-type function, we set $F_{\mu}^{-1}$ by right-continuity at $0$ and we set $F_\mu^{-1}(1)=1$.
According to Vallender \cite{Vallender}, the $1$-Wasserstein distance of $\mu,\nu\in\mathcal{W}_1(Y)$ can be calculated by
\begin{equation} \label{eq:vall-central}
d_{\mathcal{W}_1}(\mu, \nu) = \int_{0}^1 |F_\mu(x)-F_\nu(x)|~\dx = \int_0^1 |F_\mu^{-1}(x)-F_\nu^{-1}(x)|~\dx
\end{equation}
for all $\mu,\nu\in\Wo(Y)$.
Therefore the map $j$ called \emph{flip}
\begin{equation}\label{eq:flip}
j\colon \, \Wo(Y)\rightarrow \Wo(Y), \qquad \mu \mapsto j(\mu), \quad F_{j(\mu)}=F_{\mu}^{-1}
\end{equation}
 is an isometry of $\Wo(Y)$.
\begin{figure}[H]
\centering
\includegraphics[width = 4in]{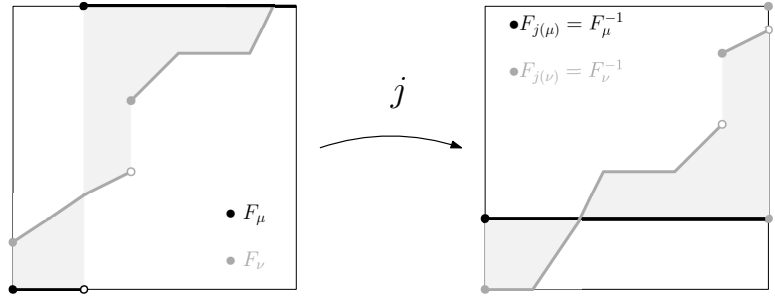}
\caption{Illustration for \eqref{eq:vall-central} and \eqref{eq:flip}.}
\end{figure}
Observe immediately that this map does not leave the set $\Delta([0,1])$ invariant
$$j(\delta_t) = t\cdot\delta_0+(1-t)\cdot\delta_1 \qquad (t\in(0,1)).$$

Now we turn to the strictly convex case $p>1$.

\begin{example}
Fix $1\leq p < \infty$, and let us equip the set $[0,1]$ with the metric $\rho(x,y) = |x-y|^{\frac{1}{p}}$. Let the symbol $X$ stand for the Polish metric space $([0,1],\rho)$. Since this metric space has finite diameter, every Borel probability measure on $[0,1]$ is automatically an element of both $\Wp(X)$ and $\Wo(Y)$. 
Notice that
\begin{align*}
    d_{\Wp(X)}^{p}(\mu,\nu) &= \inf_{\pi\in \Pi(\mu,\nu)}\int\limits_{{[0,1]\times[0,1]}}\rho(x,y)^p \;\mathrm{d}\pi(x,y) \\ 
    &= \inf_{\pi\in \Pi(\mu,\nu)}\int\limits_{{[0,1]\times[0,1]}}\big|x-y\big|\; \mathrm{d}\pi(x,y)= d_{\Wo(Y)}(\mu,\nu) \qquad (\mu,\nu\in\prob([0,1])),
\end{align*}
and therefore the map $j$ is also an isometry of $\Wp(X)$.
\end{example}
The isometry $j$ does not preserve the set of Dirac measures thus it cannot be a push-forward of any isometry of $X$. Consequently, this example answers Kloeckner's problems affirmatively.

We note that $j$ demonstrates also that an isometry does not preserve the existence of a transport map between measures in general. Indeed, for any $0<s<t<1$ there is a transport map from $\delta_s$ to $\delta_t$ but there is no such map from $j(\delta_s)$ to $j(\delta_t)$.

We close this section by mentioning that as a consequence of Theorem \ref{thm:concave-cost}, the answers to \cite[Questions 1--2]{Kloeckner-2010} are negative in the $0<p<1$ case.

\section{A complete characterization of isometries of $\mathcal{W}_p(E)$ for $1\leq p<\infty$}\label{s:pgeq1}

In this section we describe the structure of isometries of $\Wp(E)$ spaces for all $1\leq p<\infty$. Since the value of $p$ affects the smoothness of the cost function, the proof is divided into three cases: when $p$ is not even, when $p=2$, and when $p=2k$ with $k\in\N$, $k\geq2$.

In light of the example presented in Section \ref{s:example}, in general it is not true that isometries of Wasserstein spaces map Dirac masses into Dirac masses.
However {as we shall see}, if $E$ is a real separable Hilbert space and $1\leq p<\infty$, then for any isometry $\Phi\colon\Wp(E)\to\Wp(E)$ there exists a $\psi\in\isom(E)$ such that $\Phi(\delta_x) = \psi_\#\delta_x$ for all $x\in E$.

First we recall two important notions.
\begin{definition}[Dilation of a measure]
	The dilation of center $x\in E$ and ratio $\lambda \in\R$ is the map
	$$
	D_x^\lambda\colon E\to E, \quad y\mapsto x + \lambda (y-x) = (1-\lambda)x + \lambda y.
	$$
	The dilation of the measure $\mu$ of center $x$ and ratio $\lambda$ is defined by $\left(D_x^\lambda\right)_\# \mu$.
\end{definition}
\begin{definition}[Geodesics]
	A geodesic is an isometric embedding $\gamma\colon I \to \Wp(E)$, that is,
	$$|s-t|=\dwp(\gamma(s),\gamma(t))\qquad(s,t\in I)$$
	where $I\subseteq\R$ is some closed (finite or infinite) interval. A geodesic is complete if it is defined on the whole real line $\R$.
A geodesic segment is a geodesic where the parameter set is $[0,T]$ for some $T>0$. In the case when the parameter set is $[0,\infty)$ we use the term \emph{geodesic ray}.
\end{definition}

As was proved in \cite[Section 7.2]{Ambrosio} and explained in \cite[Subsection 2.1]{Kloeckner-2010}, there is a one-to-one correspondence between optimal couplings of $\mu$ and $\nu$ and geodesics connecting them, provided that $1<p<\infty$. The explicit statement reads as follows.

\begin{lemma}\label{lem:geod}
	Let $1<p<\infty$, $\mu,\nu\in\Wp(E)$ and $T=\dwp(\mu,\nu)$. 
	Define the map
	$$
	g_s\colon E\times E \to E, \quad g_s(x,y)=(1-s)x+sy
	$$
	for all $0\leq s\leq 1$.
	For any $\pi\in\Pi^0(\mu,\nu)$ optimal coupling, the curve 
	\begin{equation}\label{eq:geod} \gamma\colon [0,T], \quad t \mapsto \left(g_{t/T}\right)_\# \pi \end{equation}
	is a geodesic connecting $\mu$ and $\nu$.
	Conversely, any geodesic between $\mu$ and $\nu$ is obtained in this way.
\end{lemma}

In case when $p=2$, Kloeckner characterised Dirac measures in terms of geodesics in \cite[Section 2.3]{Kloeckner-2010}. The following lemma says that the same characterisation holds true for all $1<p<\infty$. Since the proof is rather standard, we relegate it {together} with the proof of Lemma \ref{lem:p=1Dirac} into the Appendix.
\begin{lemma}\label{lem:p>1Dirac}
	Let $1<p<\infty$ and assume that $\dim E \geq 2$. For a measure $\mu\in\Wp(E)$ the following statements are equivalent:
	\begin{itemize}
	    \item[(i)] $\mu$ is a Dirac measure,
	    \item[(ii)] any geodesic segment $\gamma:[0,T]\to\Wp(E)$ issued from $\mu$ (i.e., $\gamma(0)=\mu$) can be extended to $[0,\infty)$.
	\end{itemize}
\end{lemma}

\noindent The following is a metric characterisation of Dirac measures for the $p=1$, $\dim E \geq 2$ case. 

\begin{lemma}\label{lem:p=1Dirac}
Let $E$ be a real separable Hilbert space such that $\dim E \geq 2$. For a measure $\mu\in\Wo(E)$ the following statements are equivalent:
\begin{itemize}
    \item[(i)] $\mu$ is a Dirac measure,
    \item[(ii)] for all $\nu\in\Wo(E),\nu\neq\mu$ there exists an $\eta\in\Wo(E)$ such that 
	\begin{equation}\label{eq:W1-Dirac}
	\dwo(\mu,\nu) = \dwo(\nu,\eta) = \frac{1}{2}\dwo(\mu,\eta).
	\end{equation}
\end{itemize}
\end{lemma}

As a consequence of Lemmas \ref{lem:p>1Dirac} and \ref{lem:p=1Dirac} we obtain that the action of an isometry on $\Delta(E)$ is induced by an isometry of $E$. This is a straightforward consequence of the metric characterization of Dirac masses and that $\dwp(\delta_x,\delta_y)=\|x-y\|$ holds for all $x,y\in E$, $1\leq p<\infty$.
\begin{corollary}\label{cor:Dirac}
Let $E$ be a separable real Hilbert space and let $1\leq p<\infty$ be fixed. For any isometry $\Phi\colon\Wp(E)\to\Wp(E)$ there exists a $\psi\in\isom(E)$ such that $\Phi(\delta_x) = \delta_{\psi(x)}$ for all $x\in E$.
\end{corollary}

\subsection{The case of $1\leq p<\infty$, $2\nmid p$ -- recovery of measures from their potentials}\label{ss:recovery}

Our goal in this subsection is to recover certain properties of the measure $\mu\in\Wp(E)$ from the following \emph{potential function}:
\begin{equation}\label{eq:pot}
    \potmu\colon E\to \R, \quad x\mapsto \dwpp(\mu,\delta_x) = \int_E \norm{x-y}^p \dd\mu(y).
\end{equation}
We shall do that by showing the following identity for all $\mu\in\Wp(E)$ and $\x \in E$:
\begin{align}
&\lim_{h \to \0}\frac{\sum_{j=0}^{2k} \binom{2k}{j}(-1)^j\potmu\ler{\x+(k-j)\h}}{\ler{\sum_{j=0}^{2k} \binom{2k}{j}(-1)^j |k-j|^p }\norm{\h}^p} = \mu(\{x\}),  \label{eq:atom}
\end{align}
where we set $k:=\lceil p/2 \rceil$.
Note that this does not hold for even positive integers $p$, since in that case the potential function itself does not contain enough information to fully identify even a finitely supported probability measure (see Subsections \ref{ss:2}--\ref{ss:468} for the details). In the next two statements we prove that the denominator of \eqref{eq:atom} is not zero.

\begin{lemma}\label{lem:alternant}
Let $m\in\N$, $m\geq 1$ and $\{\lambda_\ell\}_{\ell=1}^m\subset\R$ not all of them zero. Furthermore, let $\{a_\ell\}_{\ell=1}^m$ be a set of $m$ pairwise different positive numbers. Then the function
\begin{equation}\label{eq:alternant}
t \mapsto \sum_{\ell=1}^{m} \lambda_\ell a_\ell^t
\end{equation}
has at most $m-1$ zeros.
\end{lemma}

\begin{proof}
	Without loss of generality we can assume that all the $\lambda_\ell$'s are non-zero.
	For $m=1$ the statement is obvious. Assume we already proved it for $m-1$ where $m\geq 2$. 
	The function in \eqref{eq:alternant} and
	\begin{equation}\label{eq:a2}
	t \mapsto 1+\sum_{\ell=1}^{m-1} \frac{\lambda_\ell}{\lambda_m} \ler{\frac{a_\ell}{a_m}}^t
	\end{equation}
	have exactly the same zeros.
	The derivative of the latter is
	$$
	t \mapsto \sum_{\ell=1}^{m-1} \log\ler{\frac{a_\ell}{a_m}} \frac{\lambda_\ell}{\lambda_m} \ler{\frac{a_\ell}{a_m}}^t.
	$$
	By our hypothesis, this derivative has at most $m-2$ zeros, which completes the proof.
\end{proof}

\begin{corollary}\label{cor:nonzdiv}
Let $0<p<\infty, 2 \nmid p$ and $k=\lceil p/2 \rceil$. Then we have
$$\sum_{j=0}^{k-1} \binom{2k}{j}(-1)^j (k-j)^p \neq 0.$$
\end{corollary}

\begin{proof}
For $k=1$, that is when $0<p<2$, the value of the sum is $1$. We assume from now on that $k\geq 2$. For a fixed such $k\in\N$ consider the function 
$$\R\to\R, \;\;\; t\mapsto \sum_{j=0}^{k-1} \binom{2k}{j}(-1)^j (k-j)^t$$
which has at most $k-1$ zeros by Lemma \ref{lem:alternant}. 
However, we claim that $t=2m$ is a zero for every $m\in\N$, $1\leq m <k$. 
Indeed, since $\ler{\sin (\pi t)}^{2k}=\mathcal{O}\ler{t^{2k}}$ as $t\to 0$, we calculate
\begin{align}
0 & 
= \frac{\dd^{2m}}{\dt^{2m}} \ler{\ler{\frac{1}{2i}\ler{e^{\pi i t}-e^{-\pi i t}}}^{2k}}\Bigg|_{t=0} =\frac{\dd^{2m}}{\dt^{2m}}\ler{ \ler{-\frac{1}{4}}^k \sum_{j=0}^{2k}\binom{2k}{j}(-1)^j e^{2 \pi i (k-j)t} }\Bigg|_{t=0} \nonumber\\
&= \frac{(2 \pi i)^{2m}}{(-4)^k} \sum_{j=0}^{2k}\binom{2k}{j}(-1)^j (k-j)^{2m}
= 2\frac{(2 \pi i)^{2m}}{(-4)^k} \sum_{j=0}^{k-1}\binom{2k}{j}(-1)^j (k-j)^{2m}. \label{eq:comb}
\end{align}
Hence the proof is done.
\end{proof}

As a next step in proving \eqref{eq:atom}, in the following lemma we construct a family of functions whose pointwise limit is the characteristic function of the origin. 
\begin{lemma}\label{lem:peak}
Let $E$ be a separable real Hilbert space, $0<p<\infty$, $2 \nmid p$, and $k=\lceil p/2 \rceil$. Then for all $\x \in E$ we have the following formula:
\begin{align}
&\lim_{h \to \0}\frac{ \sum_{j=0}^{2k} \binom{2k}{j}(-1)^j\norm{\x+(k-j)\h}^p}{\ler{\sum_{j=0}^{2k} \binom{2k}{j}(-1)^j |k-j|^p} \norm{\h}^p}=\bigchi_{\{\0\}}(x).  \label{eq:peak}
\end{align}
\end{lemma}

\begin{proof} In case when $\x=\0$, we easily see that the value of the limit in \eqref{eq:peak} is $1$.
Assume from now on that $\x \neq \0$. The numerator of \eqref{eq:peak} is
\begin{align*}
\binom{2k}{k} &(-1)^k \norm{\x}^p+\sum_{j=0}^{k-1} \binom{2k}{j}(-1)^j\ler{\norm{\x+(k-j)\h}^p+\norm{\x-(k-j)\h}^p}\\ 
&= \binom{2k}{k} (-1)^k \norm{\x}^p
+\sum_{j=0}^{k-1} \binom{2k}{j}(-1)^j\Bigg(\ler{\norm{\x}^2+2(k-j)\inner{\x}{\h}+(k-j)^2\norm{\h}^2}^\frac{p}{2}+\\
& \hspace{6.3cm}+\ler{\norm{\x}^2-2(k-j)\inner{\x}{\h}+(k-j)^2\norm{\h}^2}^\frac{p}{2}\Bigg) \\
& =
\norm{\x}^p
\Bigg[ \binom{2k}{k} (-1)^k
+\sum_{j=0}^{k-1} \binom{2k}{j}(-1)^j \bigg\{
\ler{1+\tfrac{2(k-j)\inner{\x}{\h}}{\norm{\x}^2}+\tfrac{(k-j)^2\norm{\h}^2}{\norm{\x}^2}}^\frac{p}{2}+ \\
& \hspace{6.3cm}+\ler{1-\tfrac{2(k-j)\inner{\x}{\h}}{\norm{\x}^2}+\tfrac{(k-j)^2\norm{\h}^2}{\norm{\x}^2}}^\frac{p}{2} \bigg\} \Bigg].
\end{align*}
It is easy to check that if $\norm{h} < \frac{\norm{x}}{5k}$, then $\left|\pm\tfrac{2(k-j)\inner{\x}{\h}}{\norm{\x}^2}+\tfrac{(k-j)^2\norm{\h}^2}{\norm{\x}^2}\right| < \frac{1}{2}$. For such vectors $h$ we compute the numerator of \eqref{eq:peak} further, using the absolute convergence of the binomial series:
\begin{align}
&\norm{\x}^p
\Bigg[
\binom{2k}{k} (-1)^k+\sum_{j=0}^{k-1} \binom{2k}{j}(-1)^j 
\Bigg\{
\sum_{l=0}^\infty \binom{\frac{p}{2}}{l}
\bigg(
\ler{\tfrac{2(k-j)\inner{\x}{\h}}{\norm{\x}^2}+\tfrac{(k-j)^2\norm{\h}^2}{\norm{\x}^2}}^l+ \nonumber \\
& \hspace{6.2cm}+\ler{-\tfrac{2(k-j)\inner{\x}{\h}}{\norm{\x}^2}+\tfrac{(k-j)^2\norm{\h}^2}{\norm{\x}^2}}^l
\bigg)
\Bigg\}
\Bigg] \label{eq:curly}
\end{align}
If we apply the binomial theorem for the $l$--powers in between the curly brackets above, then we obtain two double series. We estimate them in the following way:
\begin{align*}
&\sum_{l=0}^\infty \sum_{i=0}^{l} \left|\binom{\frac{p}{2}}{l} \binom{l}{i}\ler{\pm\tfrac{2(k-j)\inner{\x}{\h}}{\norm{\x}^2}}^{i}\ler{\tfrac{(k-j)^2\norm{\h}^2}{\norm{\x}^2}}^{l-i}\right| \\
&\leq\lesq{\sum_{l=0}^{2k-1} \sum_{\substack{i=0 \\ 2l-i \leq 2k-1}}^{l} + \sum_{l=0}^{2k-1} \sum_{\substack{i=0 \\ 2l-i > 2k-1}}^{l} + \sum_{l=2k}^\infty \sum_{i=0}^{l}} \left|\binom{\frac{p}{2}}{l}\right| \binom{l}{i} 2^i(k-j)^{2l-i}
\ler{\tfrac{\norm{\h}}{\norm{\x}}}^{2l-i}\\
&\leq \sum_{l=0}^{2k-1} \sum_{\substack{i=0 \\ 2l-i \leq 2k-1}}^{l} \left|\binom{\frac{p}{2}}{l}\right| \binom{l}{i} 2^i(k-j)^{2l-i}
\ler{\tfrac{\norm{\h}}{\norm{\x}}}^{2l-i} + \mathcal{O}\ler{\norm{h}^{2k}}+\\
&\qquad\qquad\qquad\qquad\qquad\qquad + \ler{\tfrac{\norm{\h}}{\norm{\x}}}^{2k} \sum_{l=2k}^\infty \sum_{i=0}^{l} \left|\binom{\frac{p}{2}}{l}\right| \binom{l}{i} (2k)^{2l-i} \ler{\tfrac{\norm{\h}}{\norm{\x}}}^{2l-i-2k}.
\end{align*}
where the error term is uniform in $h$ for $\norm{h} < \frac{\norm{x}}{5k}$.  Since the last summand can be further estimated from above by
\begin{align*}
\ler{\tfrac{\norm{\h}}{\norm{\x}}}^{2k} \sum_{l=2k}^\infty \left|\binom{\frac{p}{2}}{l}\right| \ler{5k}^{2k} \sum_{i=0}^{l} \binom{l}{i} \ler{\tfrac{2}{5}}^{2l-i}\leq \ler{\tfrac{\norm{\h}}{\norm{\x}}}^{2k} \ler{5k}^{2k} \sum_{l=2k}^\infty \left|\binom{\frac{p}{2}}{l}\right| 2^l \ler{\tfrac{2}{5}}^{l} = \mathcal{O}\ler{\norm{\h}^{2k}},
\end{align*}
the expression in between the curly brackets in \eqref{eq:curly} is of the form
\begin{align*}
\sum_{l=0}^\infty \binom{\frac{p}{2}}{l}&
\bigg(
\ler{\tfrac{2(k-j)\inner{\x}{\h}}{\norm{\x}^2}+\tfrac{(k-j)^2\norm{\h}^2}{\norm{\x}^2}}^l +\ler{-\tfrac{2(k-j)\inner{\x}{\h}}{\norm{\x}^2}+\tfrac{(k-j)^2\norm{\h}^2}{\norm{\x}^2}}^l
\bigg) \\
&= \sum_{l=0}^{2k-1} \sum_{\substack{i=0 \\ 2l-i \leq 2k-1 \\ i\equiv 0\; (mod \;2)}}^{l} 2\binom{\frac{p}{2}}{l} \binom{l}{i}\ler{\tfrac{2(k-j)\inner{\x}{\h}}{\norm{\x}^2}}^{i}\ler{\tfrac{(k-j)^2\norm{\h}^2}{\norm{\x}^2}}^{l-i} + \mathcal{O}\ler{\norm{\h}^{2k}} \\
&= 2+ \sum_{l=1}^{2k-1} \sum_{\substack{i=0 \\ 2l-i \leq 2k-1 \\ i\equiv 0\; (mod \;2)}}^{l} \binom{\frac{p}{2}}{l} \binom{l}{i} 2^i(k-j)^{2l-i} \ler{\tfrac{\inner{\x}{\h}}{\norm{\x}^2}}^{i}\ler{\tfrac{\norm{\h}^2}{\norm{\x}^2}}^{l-i} + \mathcal{O}\ler{\norm{\h}^{2k}}
\end{align*}
where again the error term is uniform in $h$ for $\norm{h} < \frac{\norm{x}}{5k}$. Note that the latter double sum is equal to zero if $k=1$, in which case one easily sees that  \eqref{eq:peak} holds indeed.
For $k\geq2$ the numerator of \eqref{eq:peak} can be further computed as follows, where we use the identity $\sum_{j=0}^{k-1}\binom{2k}{j}(-1)^j (k-j)^{2m} = 0$ $(m=1,\dots, k-1)$ proved in \eqref{eq:comb}:
\begin{align*}
&\norm{\x}^p
\Bigg[
\binom{2k}{k} (-1)^k+\sum_{j=0}^{k-1} \binom{2k}{j}(-1)^j\times\\
&\qquad\times\Bigg\{
2+\sum_{l=1}^{2k-1} \sum_{\substack{i=0 \\ 2l-i \leq 2k-1 \\ i\equiv 0\; (mod \;2)}}^{l} \binom{\frac{p}{2}}{l} \binom{l}{i} 2^i(k-j)^{2l-i} \ler{\tfrac{\inner{\x}{\h}}{\norm{\x}^2}}^{i}\ler{\tfrac{\norm{\h}^2}{\norm{\x}^2}}^{l-i}+ \mathcal{O}\ler{\norm{\h}^{2k}}
\Bigg\}
\Bigg]
\end{align*}
\begin{align*}
&=\norm{\x}^p
\sum_{j=0}^{k-1} \binom{2k}{j}(-1)^j\times\\
&\qquad\times\Bigg\{ \sum_{l=1}^{2k-1} \sum_{\substack{i=0 \\ 2l-i \leq 2k-2 \\ i\equiv 0\; (mod \;2)}}^{l} (k-j)^{2l-i}\binom{\frac{p}{2}}{l} \binom{l}{i} 2^i \ler{\tfrac{\inner{\x}{\h}}{\norm{\x}^2}}^{i}\ler{\tfrac{\norm{\h}^2}{\norm{\x}^2}}^{l-i} + \mathcal{O}\ler{\norm{\h}^{2k}} \Bigg\}\\
&=\norm{\x}^p
\sum_{j=0}^{k-1} \binom{2k}{j}(-1)^j \mathcal{O}\ler{\norm{\h}^{2k}}
= \mathcal{O}\ler{\norm{\h}^{2k}}.
\end{align*}
This estimation is uniform in $h$ for $\norm{h} < \frac{\norm{x}}{5k}$.
As $2k>p$, the left hand side of \eqref{eq:peak} is $0$ for $\x \neq \0.$
\end{proof}
We point out that in case when $p$ is an even positive integer, then using \eqref{eq:comb} one can calculate that the denominator and numerator in \eqref{eq:peak} actually coincide, hence the limit in \eqref{eq:atom} is 1.

Now, we are in the position to prove the main theorem of this section.

\begin{theorem}\label{thm:noteven}
	Let $E$ be a separable real Hilbert space and $1\leq p<\infty$ such that $p$ is not an even integer.
	Assume that $\Phi\colon\Wp(E)\to\Wp(E)$ is an isometry. Then there exists an (affine) isometry $\psi\in\isom(E)$ such that
	\begin{equation}\label{eq:noteven}
	\Phi(\mu) = \psi_\#\mu \qquad (\mu\in\Wp(E)).
	\end{equation}
\end{theorem}

\begin{proof}
Using the notation of Corollary \ref{cor:Dirac} we have that the map $\widetilde{\Phi}:=\psi^{-1}_{\#}\circ\Phi$ fixes all Dirac measures. Once we show that $\widetilde{\Phi}$ fixes all elements of $\Wp(E)$, we get $\Phi=\psi_\#$. Hence it suffices to prove that if $\Phi$ itself fixes all Dirac measures then it fixes every measure. Note that this assumption implies $\potmu(x) = \potphimu(x)$ for all $\mu\in\Wp(E)$ and $x\in E$. 
Therefore it is enough to prove \eqref{eq:atom}, as it immediately implies that
\begin{equation*}\Phi(\mu)(\{x\}) = \mu(\{x\}) \qquad (\mu\in\Wp(E), x\in E),
\end{equation*}
hence that $\Phi$ fixes all measures with finite support, and thus by continuity that it fixes all measures.

To prove \eqref{eq:atom} define the function
$$
G \colon E\times E \to \R, \quad\; (x,h) \mapsto \left\{ \begin{matrix}
\frac{ \sum_{j=0}^{2k} \binom{2k}{j}(-1)^j\norm{\x+(k-j)\h}^p}{\ler{\sum_{j=0}^{2k} \binom{2k}{j}(-1)^j |k-j|^p }\norm{\h}^p} & \text{if}\; h\neq 0 \\
\bigchi_{\{0\}}(x) & \text{if}\; h=0
\end{matrix}
\right.,
$$
where we endow $E\times E$ with the natural $\ell^2$-summed norm, i.e. $\norm{(x,h)} = \sqrt{\norm{x}^2+\norm{h}^2}$.
We proved in Lemma \ref{lem:peak} that for all fixed $x\in E$ we have 
\begin{equation}\label{eq:lim0}
\lim_{h\to 0} G(x,h) = G(x,0). 
\end{equation}

We claim that the function $G$ is bounded on $E\times E$. In order to see that we use symmetry properties of $G$. Namely, observe first that
\begin{equation}\label{eq:hom}
G(tx,th) = G(x,h) \qquad (x,h\in E, t\in\R, t\neq 0).
\end{equation}
Hence it is enough to show boundedness on the set $\{(x,h)\colon\|(x,h)\|=1\}$.
Second, notice that for all linear isometries $R\colon E\to E$ we have
$$
G(Rx,Rh) = G(x,h) \qquad (x,h\in E).
$$
Fix a unit vector $e\in E$. Clearly, it suffices to prove boundedness for pairs $(\lambda e,h)$ such that $\lambda\in\R$, $|\lambda|^2+\norm{h}^2=1$.
Third, denote by $F$ the orthogonal complement of the linear subspace $\R\cdot e$. It is apparent that for all linear isometries $Q\colon F\to F$ we have
$$
G(\lambda e,\alpha e + Q z) = G(\lambda e,\alpha e + z) \qquad (\lambda,\alpha \in\R, z\in F, |\lambda|^2+|\alpha|^2+\norm{z}^2=1).
$$
Fix a unit vector $f\in F$. Plainly, it is enough to show boundedness on the following subset: 
$$
\left\{
(\lambda e,\alpha e + \beta f) \colon \lambda,\alpha,\beta\in\R, \;  \sqrt{|\lambda|^2+|\alpha|^2+|\beta|^2} = 1
\right\}.
$$
Now, we use \eqref{eq:hom} to conclude that it suffices to prove boundedness of $G$ on the following subset:
$$
\mathcal{C} := 
\left\{
(\lambda e,\alpha e + \beta f) \colon \lambda,\alpha,\beta\in\R, \;  \max\{|\lambda|,|\alpha|,|\beta|\} = 1
\right\}. 
$$
It is apparent that $\mathcal{C}$ is compact in $E\times E$, and that $G$ is continuous on $\mathcal{C}$. For the latter, we see this on $\mathcal{C}\setminus\{(-e,0),(e,0)\}$ simply from the definition of $G$, and at the points $(\pm e,0)$ by \eqref{eq:lim0}.

Now we use the Lebesgue dominant convergence theorem to complete the proof. Namely, we calculate the left-hand side of \eqref{eq:atom} as follows:
\begin{align*}
\lim_{h \to \0}  \int_E\frac{\sum_{j=0}^{2k} \binom{2k}{j}(-1)^j
 \norm{\x+(k-j)\h-y}^p }{\ler{\sum_{j=0}^{2k} \binom{2k}{j}(-1)^j |k-j|^p }\norm{\h}^p}\dd\mu(y)&=\lim_{\h \to \0} \int_E G(x-y,h) \dd\mu(y)\\
&= \int_E G(x-y,0) \dd\mu(y) = \mu(\{x\}).
\end{align*}
	
\end{proof}


\subsection{The case of $p=2$ -- existence of nontrivial isometries}\label{ss:2}
In \cite[Theorems 1.1--1.2]{Kloeckner-2010} Kloeckner proved a characterisation of isometries of $\Wt(E)$ for all finite dimensional Hilbert spaces $E$. 
Note that his proof does not work in the infinite dimensional case, as he uses absolutely continuous measures that have no analogue in infinite dimension. In this subsection we prove the characterisation of $\isom(\Wt(E))$ for the infinite dimensional case using finitely supported measures and building on the finite dimensional characterisation.

\begin{definition}[Barycenter]\label{def:bari}
	Let $E$ be an infinite dimensional separable Hilbert space and $\mu\in\Wt(E)$. Then the barycenter of $\mu$ is the point $m(\mu)\in E$ such that 
	\begin{equation}\label{eq:bari}
	\left\langle m(\mu), z \right\rangle = \int_E \langle x,z\rangle \dd\mu(x)
	\end{equation}
	holds for all $z\in E$.
\end{definition}

The following simple observation will play an important role later. If $\Phi\Wt(E)\to\Wt(E)$ is an isometry such that $\Phi(\delta_x)=\delta_x$  for all $x\in E$, then $\Phi$ preserves the barycenter of measures, that is, $m\ler{\Phi(\mu)} = m(\mu)$ for all $\mu\in\Wt(E)$. In order to see this, we calculate the following for all $z\in E$:
	\begin{align}\label{barycenter-isom}
	\dwt^2(\mu,\delta_{m(\mu)+z}) &= \int_E \norm{m(\mu)+z-y}^2\dd\mu(y)= \norm{z}^2-\norm{m(\mu)}^2 + \int_E \norm{y}^2\dd\mu(y).
	\end{align}
	Clearly, the minimum of the function $x\mapsto\dwt(\mu,\delta_{x}) = \dwt(\Phi(\mu),\delta_{x})$ is attained at $x=m(\mu)=m\ler{\Phi(\mu)}$.
Note also that the affine subspace spanned by $\supp(\mu)$ must contain $m(\mu)$.
\begin{definition}[Translation of a measure by a vector]
	Let $\mu\in\Wt(E)$ and $v\in E$. The map $t_v\colon E \to E$, $x\mapsto x+v$ is called the translation by $v$.
	The translation of $\mu$ by $v$ is the measure $(t_v)_\#\mu\in\Wt(E)$. Note that $\supp((t_v)_\#\mu) = t_v[\supp(\mu)]$.
\end{definition}

First we have to understand how translation affects the $\Wt$--distance (for the proof see the Appendix).

\begin{lemma}\label{lem:transl}
	Let $\mu,\nu\in\Wt(E)$ and $v\in E$. Then we have
	\begin{equation}\label{eq:transl}
	\dwt^2\left((t_v)_\#\mu,\nu\right) = \dwt^2(\mu,\nu) + \left\langle v, v + 2m(\mu) - 2m(\nu) \right\rangle.
	\end{equation}
	In particular, substituting $v = m(\nu) - m(\mu)$ gives
	\begin{equation}\label{eq:bari-transl}
	\dwt^2(\mu,\nu) = \dwt^2\left((t_{- m(\mu)})_\#\mu,(t_{- m(\nu)})_\#\nu\right) + \norm{m(\nu) - m(\mu)}^2.
	\end{equation}
	As a consequence, $\nu$ is a translated version of $\mu$ if and only if $\dwt(\mu,\nu) = \norm{m(\nu) - m(\mu)}$.
\end{lemma}

The above lemma readily implies the following analogue of \cite[Proposition 6.1]{Kloeckner-2010}, namely, that \lq\lq rotating'' around the barycenters of measures preserves the quadratic Wasserstein distance. 

\begin{corollary}
	Suppose that $E$ is an infinite dimensional separable Hilbert space.
	Let $R\colon E\to E$ be a linear isometry of $E$. Then the following map defines an isometry of $\Wt(E)$:
	\begin{equation*}\label{eq:Wt-rot}
	\Phi\colon \Wt(E) \to \Wt(E), \qquad \mu \mapsto \ler{t_{m(\mu)}}_\#\ler{R_\# \ler{\ler{t_{-m(\mu)}}_\#\mu}} = \ler{t_{m(\mu)}\circ R \circ t_{-m(\mu)}}_\#\mu.
	\end{equation*}
\end{corollary}

We continue with an analogue of \cite[Lemma 6.2]{Kloeckner-2010}. Although the argument is similar, there are some technical differences, and thus we present the proof in the Appendix. We note that affine and linear subspaces are implicitly meant to be closed.

\begin{lemma}\label{lem:ort}
	Let $\mu,\nu\in\Wt(E)$, $\sigma:=\dwt(\mu,\delta_{m(\mu)})$ and $\rho:=\dwt(\nu,\delta_{m(\nu)})$. Then
	\begin{equation}\label{eq:ort}
	\dwt^2(\mu,\nu) = \norm{m(\mu)-m(\nu)}^2+\sigma^2+\rho^2
	\end{equation}
	holds if and only if there exists two orthogonal affine subspaces $L$ and $M$ such that $\supp(\mu)\subset L$ and $\supp(\nu)\subset M$.
\end{lemma}

Now, we are in the position to prove the infinite-dimensional version of Kloeckner's result \cite[Theorems 1.1--1.2]{Kloeckner-2010}.

\begin{theorem}\label{thm:2}
	Suppose that $E$ is an infinite dimensional separable real Hilbert space and $\Phi$ is an isometry of $\Wt(E)$. Then there exists an (affine) isometry $\psi\in\isom(E)$ and a linear isometry $R\colon E\to E$ such that
	\begin{equation}\label{eq:2}
	\Phi(\mu) = \ler{\psi \circ t_{m(\mu)}\circ R \circ t_{-m(\mu)}}_\#\mu \qquad (\mu\in\Wt(E)).
	\end{equation}
\end{theorem}

\begin{proof}
	By Corollary \ref{cor:Dirac} we can assume that $\Phi(\delta_x) = \delta_x$ holds for all $x\in E$. With this assumption we have that $\Phi$ preserves the barycenter of measures (see \eqref{barycenter-isom} above), that is $m\ler{\Phi(\mu)} = m(\mu)$ for all $\mu\in\Wt(E)$, and that $\psi$ is the identity of $E$ in \eqref{eq:2}. For any $x\in E$ we use the notation 
	\begin{equation}\Wt^x(E) := \{\mu\in\Wt(E)\colon m(\mu)=x\}.
	\end{equation}
	It is clear that we have $\Phi(\Wt^x(E))=\Wt^x(E)$, and that the map 
	\begin{equation}\label{eq:Wtx-restr}
	\ler{\ler{(t_{-x})_\#}\circ\ler{\Phi|_{\Wt^x(E)}}\circ\ler{(t_{x})_\#}}\Big|_{\Wt^0(E)}
	\end{equation}
	is an isometry of $\Wt^0(E)$. Observe that by Lemma \ref{lem:transl} the above map is independent of $x$. In particular, the restriction $\Phi|_{\Wt^0(E)}$ determines $\Phi$ and vica-versa. Therefore, in order to verify \eqref{eq:2} it is enough to show that 
	\begin{equation}\label{W2null}
	\Phi(\mu)=R_\#\mu\qquad\big(\mu\in\Wt^0(E)\big)
	\end{equation}for some linear isometry $R$. Next, for every linear subspace $M\subset E$ set 
	\begin{equation}\Wt^0(M) := \{\mu\in\Wt^0(E)\colon \supp(\mu)\subset M\}.
	\end{equation}
	We say that two measures $\mu,\nu\in\Wt^0(E)$ are orthogonally supported, if their support span two orthogonal linear subspaces. By Lemma \ref{lem:ort}, the property of being orthogonally supported is preserved in both directions by the restriction $\Phi|_{\Wt^0(E)}$. For every one-dimensional linear subspace $L$ let us fix a measure $\mu_L \in \Wt^0(E)\setminus\Delta(E)$ such that $\supp(\mu_L)\subset L$. Let $M_L$ be the linear subspace generated by $\supp(\Phi(\mu_L))$, notice that $M_L\neq \{0\}$. 
	It is clear by the orthogonality-preservation that if $\{L_j\}_{j=1}^\infty$ is a complete set of pairwise orthogonal one-dimensional linear subspaces, then the subspaces $\{M_{L_j}\}_{j=1}^\infty$ are pairwise orthogonal and they also span $E$. Again by the orthogonality-preservation property, we get that 
	$$\Phi(\Wt^0(L)) \subset \Wt^0(M_L)\quad\mbox{and}\quad\Phi^{-1}(\Wt^0(M_L)) \subset \Wt^0(L)$$ 
	hold for all one-dimensional subspaces $L$, hence $\Phi^{-1}(\Wt^0(M_L)) = \Wt^0(L)$. Consequently, we always have $\dim M_{L} = 1$, since otherwise there would exist two measures in $\Wt^0(M_{L})\setminus\Delta(E)$ supported on orthogonal linear subspaces of $M_L$ whose $\Phi^{-1}$--images would be orthogonal, which is impossible as they are in $\Wt^0(L)\setminus\Delta(E)$.
	
	Now, by Uhlhorn's theorem \cite{U}, we obtain that there is a bijective linear isometry $U\colon E\to E$ such that $M_L = UL$ for all one-dimensional linear subspaces $L$. Again, by the orthogonality preservation property, we obtain that $\Phi(\Wt^0(M)) = \Wt^0(UM)$ for all linear subspaces $M\subset E$.

It was proved by Kloeckner in \cite[Theorem 1.2]{Kloeckner-2010} that if $E$ is a euclidean space with $2\leq \dim E <\infty$, and $\Phi$ is an isometry  then $\Phi|_{\Wt^0(E)} = \ler{V_\#}|_{\Wt^0(E)}$ with some linear isometry $V\colon E\to E$. Using this result, one easily obtains the following: for every linear subspace $M$, $2\leq \dim M <\infty$ there exists a bijective linear isometry $V_M\colon UM\to UM$ such that
	$$
	\Phi(\mu) = (V_M\circ U)_\#\mu \qquad (\mu\in\Wt^0(M)).
	$$
	Since all these isometries are compatible in the sense that $V_M|_{U(M\cap N)}=V_N|_{U(M\cap N)}$ for all linear subspaces $M$ and $N$ ($2\leq\dim{M},\dim{N}<\infty$), a standard argument shows that these $V_M$'s have a joint extension $V\colon E\to E$. In particular,
	$\Phi(\mu) = (V\circ U)_\#\mu$ for every finitely supported measure $\mu\in\Wt^0(E)$. Since every element of $\Wt^0(E)$ can be approximated with such measures and $\Phi$ is continuous, we get \eqref{W2null} with $R=V\circ U$.
	\end{proof}


\subsection{The case of $2\mid p$, $4\leq p$ -- isometric rigidity}
\label{ss:468}
As it was mentioned earlier, even parameters must be handled separately because in that case the potential function alone does not carry enough information to completely identify measures. To explain the difficulties better, and to highlight the main ideas of the proof, let us begin with sketching the special case $p=4$, $E=\mathbb{R}^2$. By Corollary \ref{cor:Dirac}, without loss of generality we may assume that our isometry $\Phi:\mathcal{W}_4(\mathbb{R}^2)\to\mathcal{W}_4(\mathbb{R}^2)$ leaves every Dirac measure fixed, and from here our aim is to show that $\Phi$ leaves every measure fixed. The expansion of $\|x-y\|^p$ takes the following form:
\begin{align*}
\norm{(x_1,x_2)-(y_1,y_2)}^4
 &= (x_1^4 + 2 x_1^2 x_2^2 + x_2^4) + (- 4 x_1^3 y_1 - 4 x_1 x_2^2 y_1 - 4 x_1^2 x_2 y_2 - 4 x_2^3 y_2)+ \nonumber \\
 &\qquad + (6 x_1^2 y_1^2 + 6 x_2^2 y_2^2 + 8 x_1 x_2 y_1 y_2 + 2 x_2^2 y_1^2 + 2 x_1^2 y_2^2)+ \nonumber \\
 &\qquad + (- 4 x_1 y_1^3 - 4 x_2 y_1^2 y_2 - 4 x_1 y_1 y_2^2 - 4 x_2 y_2^3) + (y_1^4 + 2 y_1^2 y_2^2 + y_2^4), 
\end{align*}
and thus the potential function $\potmuf(x_1,x_2)$ is a polynomial of degree four, where the coefficients are integrals of polynomials of $(y_1,y_2)$ with respect to $\dd\mu(y_1,y_2)$. In particular, the coefficients of $x_1^2$, $x_2^2$ and $x_1x_2$ are 
$$\int_{\R^2} 6 y_1^2 + 2 y_2^2 \dd\mu(y_1,y_2),\quad \int_{\R^2} 2 y_1^2 + 6 y_2^2 \dd\mu(y_1,y_2)\quad\mbox{and}\quad \int_{\R^2} 8 y_1 y_2 \dd\mu(y_1,y_2).$$
Therefore, if the potential functions of $\mu$ and $\nu$ coincide, then we obtain
$$
\int_{\R^2} \inner{(y_1,y_2)}{(v_1,v_2)}^2 \dd\mu(y_1,y_2) = \int_{\R^2} \inner{(y_1,y_2)}{(v_1,v_2)}^2 \dd\nu(y_1,y_2) \qquad ((v_1,v_2)\in\R^2).
$$
Hence $\mu$ is supported on a one-dimensional linear subspace if and only if the above integral is $0$ for some $(v_1,v_2)\in\R^2\setminus\{(0,0)\}$. This happens if and only if $\nu$ is supported on the same one-dimensional linear subspace.

Now, we obtain that for every one-dimensional linear subspace $L$ of $\R^2$ the isometry $\Phi$ maps $\mathcal{W}_4(L)$ bijectively onto itself. So we can use the result \cite[Theorem 3.16]{GTV} to obtain that $\Phi$ fixes all elements of $\mathcal{W}_4(L)$. In particular, it fixes all measures which are supported on two points whose affine hull contains $(0,0)$. 
From here, by Lemma \ref{lem:bisector} below, we easily obtain 
$\mu(\ell) = \Phi(\mu)(\ell)$ for every one-dimensional affine subspace $\ell$ and $\mu\in\mathcal{W}_4(\R^2)$.
It follows that that $\Phi$ fixes all finitely supported measures, and therefore by continuity $\Phi$ is the identity map on $\mathcal{W}_4(\R^2)$.\\

After this short sketch we continue with the general case, i.e. if $E$ is a separable real Hilbert space and $p=2k$ for some $k\in\mathbb{N}$, $k\geq2$. We define the following measures:
\begin{equation}\label{eq:2-p}
\zeta_{a,b}^\alpha(x):=\alpha\cdot\delta_{ax}+(1-\alpha)\cdot\delta_{bx}
\end{equation}
where $0\leq \alpha\leq 1$, $x\in E, x\neq 0$, $a,b\in\R, a\neq b$. For any two points $x,y\in E, x\neq y$ define the bisector 
$$B(x,y) := \{z\in E \colon \norm{x-z}=\norm{y-z}\} = \{x-y\}^\perp+\frac{x+y}{2}$$ 
which is an affine hyperplane. The next lemma holds for any $0<p<\infty$, its proof is given in the Appendix.

\begin{lemma}\label{lem:bisector}
	Let $0<p<\infty$, $\mu\in\Wp(E)$, $x\in E, x\neq 0$, $a,b\in\R, a\neq b$. Set
	$$m := \min\{\dwp(\mu,\zeta_{a,b}^{\alpha}(x))\colon 0\leq \alpha\leq 1 \}.$$
	Then we have
	\begin{align}
	&\mu\ler{B(ax,bx)}= \max\{\alpha\colon \dwp(\mu,\zeta_{a,b}^{\alpha}(x))=m\} - \min\{\alpha\colon \dwp(\mu,\zeta_{a,b}^{\alpha}(x))=m\}. \label{eq:bisector}
	\end{align}
\end{lemma}

Now, we prove the main theorem of this section.

\begin{theorem}\label{thm:468}
	Let $E$ be a separable real Hilbert space and $p$ be an even integer with $p\geq 4$. 
	Assume that $\Phi\colon\Wp(E)\to\Wp(E)$ is an isometry. Then there exists an (affine) isometry $\psi\in\isom(E)$ such that
	\begin{equation}\label{eq:468}
	\Phi(\mu) = \psi_\#\mu \qquad (\mu\in\Wp(E)).
	\end{equation}
\end{theorem}

\begin{proof}
Again by Corollary \ref{cor:Dirac}, we can assume without loss of generality that all Dirac measures are fixed. 
Observe that 
\begin{align*}
\norm{x-y}^p = \ler{\norm{x}^2-2\inner{x}{y}+\norm{y}^2}^k= \sum_{i,j,\ell\in\N, i+j+\ell = k} \binom{k}{i, j, \ell}(-2)^i\inner{x}{y}^i \norm{y}^{2j}\norm{x}^{2\ell},
\end{align*}
where $\binom{k}{i,j,\ell} = \frac{k!}{i!\,j!\,\ell!}$ is the trinomial coefficient. The potential function has the form
\begin{align*}
\potmu(x) = \sum_{\substack{i,j,\ell\in\N, \\ i+j+\ell = k}} \binom{k}{i, j, \ell}(-2)^i\norm{x}^{2\ell} \int_E \inner{x}{y}^i \norm{y}^{2j} \dd\mu(y).
\end{align*}
In particular, $\potmu(0) = \int_E \norm{y}^{2k} \dd\mu(y)$ and its derivative at $x=0$ is the bounded linear functional
$$
D\potmu(0)\colon E\to \R, \qquad h\mapsto \binom{k}{1, k-1, 0}(-2) \int_E \inner{h}{y} \norm{y}^{2k-2} \dd\mu(y).
$$
Indeed the property $\potmu(x) = \potmu(0) + D\potmu(0)x + \mathcal{O}(\norm{x}^2)$ is easily seen from the Cauchy--Schwartz inequality.
The $\mathcal{O}(\norm{x}^2)$ term is
\begin{align*}
\potmu(x) - \potmu(0) - D\potmu(0)x =\sum_{\substack{i,j,\ell\in\N, \\ i+j+\ell = k, \; i+2\ell\geq 2}}
\binom{k}{i, j, \ell}(-2)^i\norm{x}^{2\ell} \int_E \inner{x}{y}^i \norm{y}^{2j} \dd\mu(y),
\end{align*}
which clearly coincides with $\potphimu(x)-\potphimu(0)-D\potphimu(0)x$. 

Notice that for any fixed vector $x\in E$ with $\norm{x}=1$ we have the following expression for the (constant multiple of the) second directional derivative along the direction $x$. We again use the Cauchy--Schwartz inequality:
\begin{align*}
\func(\mu,x) :&= \lim_{t\to 0+} \frac{1}{t^2}\ler{\potmu(tx) - \potmu(0) - D\potmu(0)tx} 
\\
&= \lim_{t\to 0+} 
\sum_{\substack{i,j,\ell\in\N, \\ i+j+\ell = k, \; i+2\ell\geq 2}}
\binom{k}{i, j, \ell}(-2)^i t^{2\ell+i-2} \int_E \inner{x}{y}^i \norm{y}^{2j} \dd\mu(y) \\
&= k \int_E \norm{y}^{2k-2} \dd\mu(y)
+ 2k(k-1) \int_E \inner{x}{y}^2 \norm{y}^{2k-4} \dd\mu(y).
\end{align*}
Note that we have $\func(\mu,x) = \func(\Phi(\mu),x)$ for all measures $\mu$ and unit vectors $x$.

Now assume that $2\leq\dim E<\infty$. Take an orthonormal base $\{e_n\}_{n=1}^{\dim E}$, and consider
\begin{align*}
\sum_{n=1}^{\dim E} \func(\mu,e_j) &= k \dim E \int_E \norm{y}^{2k-2} \dd\mu(y)
+ 2k(k-1) \sum_{n=1}^{\dim E} \int_E \inner{e_j}{y}^2 \norm{y}^{2k-4} \dd\mu(y) \\
&= \ler{k \dim E + 2k(k-1)} \int_E \norm{y}^{2k-2} \dd\mu(y).
\end{align*}
This shows that $$\int_E \norm{y}^{2k-2} \dd\Phi(\mu)(y) = \int_E \norm{y}^{2k-2} \dd\mu(y),$$ hence 
\begin{align}\label{eq:inner2}
\int_E \inner{x}{y}^2 \norm{y}^{2k-4} \dd\Phi(\mu)(y) = \int_E \inner{x}{y}^2 \norm{y}^{2k-4} \dd\mu(y)
\end{align}
holds for all $\mu\in\Wp(E)$ and $x\in E$ such that $\|x\|=1$.
On the other hand, if $\dim E = \infty$, then we again consider an orthonormal base $\{e_n\}_{n=1}^{\infty}$, and take the limit:
\begin{align*}
\lim_{n\to\infty}\func(\mu,e_n) & = k \int_E \norm{y}^{2k-2} \dd\mu(y)
+ 2k(k-1) \lim_{n\to\infty} \int_E \inner{e_n}{y}^2 \norm{y}^{2k-4} \dd\mu(y) \\
& = k \int_E \norm{y}^{2k-2} \dd\mu(y),
\end{align*}
where we used the Cauchy--Schwartz inequality and Lebesgue's majorant convergence theorem. Therefore we obtain \eqref{eq:inner2} for this case too.

Now, $\mu$ being supported on the linear subspace $\{x\}^\perp$ is equivalent to saying that the expression in \eqref{eq:inner2} is zero. But this holds if and only if $\Phi(\mu)$ is supported on $\{x\}^\perp$. If we consider this property for an orthonormal basis, we easily infer that $\Phi$ maps $\Wp(L)$ bijectively onto itself for every one-dimensional linear subspace $L\subset E$. By \cite[Theorem 3.16]{GTV} we conclude that the restriction $\Phi|_{\Wp(L)}$ is the identity map, and thus that $\Phi$ fixes all measures which are supported on two points whose affine hull contains $0$. In other words, $\Phi(\zeta_{a,b}^{\alpha}(x))=\zeta_{a,b}^{\alpha}(x)$ holds true for all $\zeta_{a,b}^{\alpha}$ defined in \eqref{eq:2-p}.
By Lemma \ref{lem:bisector} and the observation made just before it, we have \begin{equation}\label{ah}
\mu\ler{H} = \Phi(\mu)(H)
\end{equation}
for every affine hyperplane $H$.  Let $\mu$ be a finitely supported measure with support $\supp(\mu)=\{x_1,\dots,x_n\}$ and denote by $d$ the dimension of the linear subspace spanned by $\supp(\mu)$. We claim that $\supp\big(\Phi(\mu)\big)\subseteq\{x_1,\dots,\x_n\}$. Indeed, as $\mu$ is finitely supported, the set $\{x_i-x_j\,|\,1\leq i\neq j\leq n\}$ is finite, and therefore there exists infinitely many $y\in E$ such that $\|y\|=1$ and $\inner{y}{x_i-x_j}\neq0$ for all $i\neq j$. Let us denote the set of such vectors by $Y$. For all $y\in Y$ we can define a collection of affine hyperplanes as follows: $H_j^{(y)}:=x_j+\{y\}^{\perp}$ ($1\leq j\leq n)$. Observe that $H_j^{(y)}\cap\supp(\mu)=\{x_j\}$ and that $H^{(y)}:=\bigcup_{j=1}^n H_j^{(y)}$ is a disjoint union of affine hyperplanes such that $\supp(\mu)\subseteq H^{(y)}$. Now it follows from \eqref{ah} that
\begin{equation}
    1=\sum_{i=1}^n\mu(\{x_i\})=\sum_{i=1}^n\mu(H_i^{(y)})=\sum_{i=1}^n\Phi(\mu)(H_i^{(y)})=\Phi(\mu)(\bigcup_{i=1}^n H_i^{(y)})=\Phi(\mu)(H^{(y)})
\end{equation}
The set $H^{(y)}$ is closed, and therefore $\supp\big(\Phi(\mu)\big)\subseteq H^{(y)}$. Since $\bigcap_{y\in Y}H^{(y)}=\{x_1,\dots,x_n\}$, we have $\supp\big(\Phi(\mu)\big)\subseteq\{x_1,\dots,x_n\}$. In fact, to obtain this, it is enough to choose a collection of linearly independent vectors $\{y_j\}_{j=1}^{d+1}$ from $Y$. From here we can finish the proof easily. Let us fix a $y\in Y$ and observe that 
\begin{equation}
\mu(\{x_j\})=\mu(H_j^{(y)})=\Phi(\mu)(H_j^{(y)})=\Phi(\mu)\Big(H_j^{(y)}\cap\supp\big(\Phi(\mu)\big)\Big)=\Phi(\mu)(\{x_j\})
\end{equation}
for all $x_j\in\supp(\mu)$. So we get $\Phi(\mu) = \mu$ for all finitely supported measures. A continuity argument then completes the proof.
\end{proof}

\section{Isometric rigidity of $\Wp(E)$ for $0<p<1$ and a more general class of Wasserstein spaces}\label{s:ple1}
The case $0<p<1$, that is when the transport cost is a concave function of the distance, is special in many regards. From the theoretical point of view, this case is interesting because the transport plans have rather different structure. From the economic point of view, this setting seems to be the most natural one when  moving a mass has a cost which is proportionally less if the distance increases. For more details about the case of strictly concave cost functions we refer the reader to the introduction of \cite{GM} (see also Section 3.3.2 in \cite{Santambrogio} and Section 2.4.4 in \cite{villani-ams-book}).

In this section we prove that $\Wp(E)$ is isometrically rigid if $0<p<1$. In fact, this will be a straightforward consequence of our more general result: $\Wo(X)$ is isometrically rigid if the metric of the underlying space $X$ satisfies the \emph{strict triangle inequality}
\begin{equation}\label{eq:stricttr}
    \rho(x,y) < \rho(x,z) + \rho(z,y) \quad (x,y,z,\in X, z\notin\{x,y\}).
\end{equation}
 As it was mentioned before, $\Wp(E)$ is basically $\Wo(X)$ where $(X,\rho) = (E,\norm{\cdot}^{p})$ and $\rho$ satisfies the strict triangle inequality, see \cite[Lemma 5.1]{GM}. 

To avoid trivialities we assume that $X$ has at least three points. The next statement is part of the folklore, however, we decided to state it here and relegate its proof into the Appendix for the reader's convenience. Briefly, it says that if the strict triangle inequality holds, then the shared weight between two measures stays in place under an optimal transport plan. 

Recall that if $\mu$ and $\nu$ are nonnegative measures, then the symbols $(\mu-\nu)_+$ and $(\mu-\nu)_-$ stand for the positive and negative parts of $\mu-\nu$, respectively, while $\mu\wedge\nu$ denotes the greatest lower bound of $\mu$ and $\nu$. For positive measures $\mu^0$ and $\nu^0$ with $\mu^0(X) = \nu^0(X)$, the symbol $\Pi(\mu^0,\nu^0)$ denotes the set of all positive measures on $X\times X$ such that their marginals are $\mu^0$ and $\nu^0$.

\begin{theorem}\label{thm:marad}
    Let $(X,\rho)$ be a complete separable metric space that satisfies the strict triangle inequality \eqref{eq:stricttr}, and denote by $D$ the diagonal $\{(x,x)\colon x\in X\}$ in $X\times X$. If $\mu,\nu\in\Wo(X)$ and $\pi\in\Pi^0(\mu,\nu)$, then
    $$\pi|_D = (\id\times\id)_\#(\mu\wedge\nu).$$ 
    In particular, if we set $\mu^0 := (\mu-\nu)_+ = \mu-(\mu\wedge\nu)$ and $\nu^0 := (\mu-\nu)_- = \nu-(\mu\wedge\nu)$, then
    \begin{equation*}
        \dwo(\mu,\nu) = \inf_{\vartheta\in\Pi(\mu^0,\nu^0)}\left\{
        \int_{X\times X} \rho(x,y) \dd\vartheta(x,y)
        \right\}.
    \end{equation*}
\end{theorem}

\begin{definition}[Metric $\lambda$--ratio set]
    Let $(X,\rho)$ be a complete separable metric space, $0<\lambda<1$, and $\mu,\nu\in\Wo(X)$. Then their metric $\lambda$--ratio set is
    $$ M_\lambda(\mu,\nu) = \left\{ \eta\in\Wo(X) \colon \dwo(\mu,\eta) = \lambda\cdot\dwo(\mu,\nu), \; \dwo(\eta,\nu) = (1-\lambda)\cdot\dwo(\mu,\nu) \right\}.$$
    The set $M_{1/2}$ is sometimes called the \emph{metric midpoint set} of $\mu$ and $\nu$.
\end{definition}

\begin{definition}[Composition/gluing of transport plans]\label{def:comp}
    Let $(X,\rho)$ be a complete separable metric space and denote by $X_1, X_2, X_3$ three identical copies of $X$. Let $\mu_j\in\prob(X_j)$ ($j=1,2,3$), and $\pi_{12}\in\Pi(\mu_1,\mu_2)$, $\pi_{23}\in\Pi(\mu_2,\mu_3)$. 
    Consider their disintegrations
    $$
    \pi_{12} = \int_{X_2} \pi_{12;2}^{(x_2)}\otimes\delta_{x_2}\dd\mu_2(x_2), \qquad \pi_{23} = \int_{X_2} \delta_{x_2}\otimes\pi_{23;2}^{(x_2)}\dd\mu_2(x_2)
    $$
    where $\pi_{12;2}\colon X_1 \to \prob(X_1)$, $\pi_{23;2}\colon X_1 \to \prob(X_3)$ are measurable mappings.
    Define the measure
    $$\ppi := \int_{X_2} \pi_{12;2}^{(x_2)}\otimes\delta_{x_2}\otimes\pi_{23;2}^{(x_2)}\dd\mu_2(x_2) \in \prob(X_1\times X_2 \times X_3), $$
    whose marginals are clearly $\pi_{12}$ and $\pi_{23}$ on $X_1\times X_2$ and $X_2 \times X_3$, respectively. We call the marginal of $\ppi$ on $X_1\times X_3$ the composition/gluing of the transport plans $\pi_{12}$ and $\pi_{23}$, in notation $\pi_{13} := \pi_{23}\circ\pi_{12}$.
    For more details, see \cite[p. 212--214]{villani-ams-book} or \cite[p. 122--123]{Ambrosio}.
\end{definition}
The following lemma plays a crucial role in the metric characterization of Dirac masses.

\begin{lemma}\label{lem:0<p<1Dirac}
	Let $(X,\rho)$ be a complete separable metric space that satisfies the strict triangle inequality \eqref{eq:stricttr}. Consider two distinct measures $\mu,\nu\in\Wo(X)$ and a $\lambda\in(0,1)$. Then the following are equivalent:
	\begin{itemize}
	    \item[(i)] the supports of both $\mu^0 = (\mu-\nu)_+$ and $\nu^0 = (\mu-\nu)_-$ are singletons, that is, $\mu^0 = t\cdot\delta_x$, $\nu^0 = t\cdot\delta_y$ with some $0<t\leq 1$ and $x\neq y$,
	    \item[(ii)] the metric $\lambda$--ratio set $M_\lambda(\mu,\nu)$ is a singleton.
	\end{itemize}
	Moreover, in this case the unique element of the metric $\lambda$--ratio set is 
	\begin{equation}\label{eq:lambdaratio}
	(1-\lambda)\cdot\mu+\lambda\cdot\nu = (\mu\wedge\nu) + (1-\lambda) t\cdot\delta_x + \lambda t\cdot\delta_y.
	\end{equation}
\end{lemma}

\begin{proof}
    (i)$\Longrightarrow$(ii): To make the presentation more transparent, we use the notation $X_1, X_2, X_3$ for three identical copies of $X$, as in Definition \ref{def:comp}. Assume that (i) holds. By Theorem \ref{thm:marad} 
    we obtain that there is only one optimal coupling between $\mu$ and $\nu$, namely, 
    $$\varpi := t\cdot\delta_{(x,y)} + (\id\times\id)_\#(\mu\wedge\nu).$$ 
    Hence, $\dwo(\mu,\nu) = t \, \rho(x,y)$.
    Consider an $\eta\in M_\lambda(\mu,\nu)$, two optimal couplings $\pi_{12}\in\Pi^0(\mu,\eta)$, $\pi_{23}\in\Pi^0(\eta,\nu)$, and their composition $\pi_{13}$. 
    We claim that $\pi_{13} = \varpi$. Indeed this can be seen by the following estimation which goes along the line of the estimation given in \cite[p. 213]{villani-ams-book}:
    \begin{align*}
       d_{\mathcal{W}_1}&(\mu,\nu) = \int_{X_1\times X_3} \rho(v_1,v_3) \dd\varpi(v_1,v_3)\leq \int_{X_1\times X_3} \rho(v_1,v_3) \dd\pi_{13}(v_1,v_3) \\
        &= \int_{X_1\times X_2 \times X_3} \rho(v_1,v_3) \dd\ppi(v_1,v_2,v_3)\leq \int_{X_1\times X_2 \times X_3} \rho(v_1,v_2) + \rho(v_2,v_3) \dd\ppi(v_1,v_2,v_3) \\
        &= \int_{X_1\times X_2} \rho(v_1,v_2) \dd\pi_{12}(v_1,v_2) + \int_{X_2 \times X_3} \rho(v_2,v_3) \dd\pi_{23}(v_2,v_3) \\
        &= \dwo(\mu,\eta) + \dwo(\eta,\nu)= \dwo(\mu,\nu).
    \end{align*}
    Since we must have equations in place of the two inequalities above, on the one hand this implies $\pi_{13} = \varpi$, as was claimed. On the other hand, by the strict triangle inequality we obtain that 
    \begin{equation}\label{eq:pi-ae}
        v_2\in\{v_1,v_3\} \quad\text{holds for}\quad \ppi\text{-a.e.}\quad (v_1,v_2,v_3). 
    \end{equation}
    
    Let us introduce the notations $\proj_{13}\colon X_1\times X_2\times X_3 \to X_1\times X_3,\, \proj_{13}(v_1,v_2,v_3) = (v_1,v_3)$, $D_{13} = \{(v_1,v_3)\in X_1\times X_3\colon v_1=v_3\}$, $D_{123} = \{(v_1,v_2,v_3)\in X_1\times X_2\times X_3\colon v_1=v_2=v_3\}.$ Note that 
    \begin{equation}\label{eq:pi-supp}
        \supp\ler{\ppi} \subset \proj_{13}^{-1}[\supp\ler{\varpi}] \subset \proj_{13}^{-1}[D_{13}\cup\{(x,y)\}.
    \end{equation}
    Utilising \eqref{eq:pi-ae}--\eqref{eq:pi-supp} we observe that 
    \begin{align*}
        t = \varpi\ler{\{(x,y)\}} = \ppi\ler{\{x\}\times X_2\times\{y\}} = \ppi\ler{\{x\}\times \{x,y\}\times\{y\}}
    \end{align*}
    and
    \begin{align*}
        1-t = \varpi(D_{13}) = \ppi\ler{\proj_{13}^{-1}[D_{13}]} = \ppi\ler{D_{123}}.
    \end{align*}
    This in turn implies that $\ppi = \widehat{\ppi} + (\id\times\id\times\id)_\#(\mu\wedge\nu)$ where $\supp(\widehat{\ppi}) \subset \{x\}\times \{x,y\}\times\{y\}$, and therefore 	$\eta=(1-\lambda)\cdot\mu+\lambda\cdot\nu$.
    
    \smallskip
    
    (ii)$\Longrightarrow$(i): In this part we use only one copy of $X$. Suppose that (i) does not hold. Consider an optimal coupling $\varpi\in\Pi^0(\mu,\nu)$ and set $D := \{(x,x)\colon x\in X\}$. There exist two distinct points $(x_1,x_2), (y_1,y_2)\in \supp(\varpi) \setminus D$. By interchanging the role of $\mu$, $\nu$ and $\lambda$, $1-\lambda$ if necessary, we may assume without loss of generality that $x_1\neq y_1$. Take two disjoint neighbourhoods $U, V\subset X$ of $x_1, y_1$, respectively.
    We clearly have
    $$
    \int_{U\times X}\rho(x,y)\dd\varpi(x,y) > 0 \quad\mbox{and}\quad \int_{U^c\times X}\rho(x,y)\dd\varpi(x,y) > 0
    $$
    where $U^c=X\setminus U$. For any pair $(\alpha,\beta)\in[0,1]\times[0,1]$ define
    $$
    \varpi_1^{\alpha,\beta} := \alpha\cdot\varpi|_{U\times X} + \beta\cdot\varpi|_{U^c\times X} \quad\mbox{and}\quad \varpi_2^{\alpha,\beta} := (1-\alpha)\cdot\varpi|_{U\times X} + (1-\beta)\cdot\varpi|_{U^c\times X}.
    $$ 
    One sees easily that there exist infinitely many pairs $(\alpha,\beta) \in (0,1)^2$ satisfying
    \begin{equation}\label{eq:tav-felez1}
        \int_{X\times X}\rho(x,y)\dd\varpi_1^{\alpha,\beta}(x,y) = \lambda\cdot\dwo(\mu,\nu)
    \end{equation}
    and
    \begin{equation}\label{eq:tav-felez2}   
        \int_{X\times X}\rho(x,y)\dd\varpi_2^{\alpha,\beta}(x,y) = (1-\lambda)\cdot\dwo(\mu,\nu).
    \end{equation}
    
    In what follows, for any such pair $(\alpha,\beta)\in(0,1)\times(0,1)$ we construct a probability measure $\eta^{\alpha,\beta} \in M_\lambda(\mu,\nu)$ and show that for distinct pairs we obtain different measures. Informally speaking, the plan $\eta^{\alpha,\beta}$ transfers only some of the mass from $\mu$ according to $\varpi_1^{\alpha,\beta}$ and leaves the rest intact. More precisely, denote by $\proj_j$ the projection map $\proj_j\colon X\times X \to X,\,(x_1,x_2)\mapsto x_j,\,(j=1,2)$ and define
    $$
    \xi := \varpi_1^{\alpha,\beta} + (\id\times\id)_\#\ler{(\proj_1)_\#\varpi_2^{\alpha,\beta}} \in \prob(X\times X)
    $$
    and
    $$
    \eta^{\alpha,\beta} := (\proj_2)_\#\xi = (\proj_2)_\#\varpi_1^{\alpha,\beta} + (\proj_1)_\#\varpi_2^{\alpha,\beta} \in \prob(X).
    $$
    Clearly, $\xi\in\Pi\ler{\mu,\eta^{\alpha,\beta}}$. Define also the measure $\zeta\in\Pi\ler{\eta^{\alpha,\beta},\nu}$ as follows:
    $$
    \zeta := (\id\times\id)_\#\ler{(\proj_2)_\#\varpi_1^{\alpha,\beta}} + \varpi_2^{\alpha,\beta} \in \prob(X\times X).
    $$
    Then $\eta^{\alpha,\beta} \in M_\lambda(\mu,\nu)$ follows from the following inequalities:
   $$\dwo\ler{\mu,\eta^{\alpha,\beta}} \leq \int_{X\times X} \rho(x,y)\dd\xi(x,y) = \int_{X\times X}\rho(x,y)\dd\varpi_1^{\alpha,\beta}(x,y) = \lambda\,\dwo(\mu,\nu),$$
    $$\dwo\ler{\eta^{\alpha,\beta},\nu} \leq \int_{X\times X} \rho(x,y)\dd\zeta(x,y) = \int_{X\times X}\rho(x,y)\dd\varpi_2^{\alpha,\beta}(x,y) = (1-\lambda)\,\dwo(\mu,\nu).$$

    Now, consider another pair $(\gamma,\delta) \in (0,1)\times(0,1)$, $(\gamma,\delta)\neq(\alpha,\beta)$ which also satisfies \eqref{eq:tav-felez1}-\eqref{eq:tav-felez2}, and assume that $\eta^{\alpha,\beta} = \eta^{\gamma,\delta}$. Our aim is to get a contradiction. Notice that $\alpha\neq\gamma$ and $\beta\neq\delta$ follow.
    Without loss of generality we may assume that $\alpha > \gamma$, which forces $\beta<\delta$. This and the very definitions of $\eta^{\alpha,\beta}$ and $\eta^{\gamma,\delta}$ imply that
    $$
    \ler{\proj_2}_\#\ler{(\alpha-\gamma)\cdot\varpi|_{U\times X} - (\delta-\beta)\cdot\varpi|_{U^c\times X} }= \ler{\proj_1}_\#\ler{(\alpha-\gamma)\cdot\varpi|_{U\times X} - (\delta-\beta)\cdot\varpi|_{U^c\times X}}.
    $$
    Dividing both sides by $\alpha-\gamma$ and setting $\varepsilon := \frac{\delta-\beta}{\alpha-\gamma}$, a simple rearrangement gives
    \begin{align*}
    \nu-\mu =(1+\varepsilon)\ler{\ler{\proj_2}_\#\ler{\varpi|_{U^c\times X}}-\ler{\proj_1}_\#\ler{\varpi|_{U^c\times X}}}.
    \end{align*}
    Notice that as a consequence the restriction $(\nu-\mu)|_U$ is a positive measure. Very similarly,
    \begin{align*}
    \nu-\mu = \ler{1+\tfrac{1}{\varepsilon}}\ler{\ler{\proj_2}_\#\ler{\varpi|_{U\times X}}-\ler{\proj_1}_\#\ler{\varpi|_{U\times X}}},
    \end{align*}
    thus the restriction $(\nu-\mu)|_{U^c}$ is a positive measure too. But this means that $\nu-\mu$ is a positive measure and thus $\mu=\nu$, since $\mu$ and $\nu$ are both probability measures, a contradiction.
\end{proof}
Lemma \ref{lem:0<p<1Dirac} gives a metric characterization of the property when two measures differ only in one atom. The following definition captures the property when two measures differ only in finitely many atoms.
\begin{definition}[Neighbouring measures]\label{def:neighbouring}
    We say that two measures $\mu,\nu\in\Wo(X)$ are \emph{neighbouring} (we denote it by $\mu\neighb\nu$), 
    if $\mu-\nu$ is a finitely supported (signed) measure.
The \emph{neighbouring set of $\mu$} is defined by $\mathcal{N}(\mu) := \{\nu\colon \mu\neighb\nu\}$.
\end{definition}
Observe that $\mu\neighb\nu$ if and only if there exists a finite sequence $\mu_0,\mu_1,\dots,\mu_n\in\Wo(X)$, $n\in\N$, $\mu = \mu_0$, $\nu = \mu_n$ such that $M_{1/2}(\mu_{j-1},\mu_{j})$ is a singleton for all $j=1,2,\dots,n$. Since this gives a metric characterization of the neighbouring relation,  $\mu\neighb\nu$ if and only if $\Phi(\mu)\neighb\Phi(\nu)$. Furthermore,
$$\mathcal{N}(\Phi(\mu))=\Phi[\mathcal{N}(\mu)]=\{\Phi(\xi)\colon\xi\in\mathcal{N}(\mu)\}.$$

Now we are in the position to prove the main result of this section, namely that $\Wo(X)$ is rigid whenever $X$ satisfies the strict triangle inequality.
\begin{theorem}\label{thm:concave-cost}
	Let $(X,\rho)$ be a complete separable metric space that satisfies the strict triangle inequality \eqref{eq:stricttr}.
	Assume that $\Phi\colon\Wo(X)\to\Wo(X)$ is an isometry. Then there exists an isometry $\psi\in\isom(X)$ such that
	\begin{equation}\label{eq:concave-cost}
	\Phi(\mu) = \psi_\#\mu \qquad (\mu\in\Wo(X)).
	\end{equation}
\end{theorem}

\begin{proof}
First observe that for any measure $\mu\in\Wo(X)$ the following assertions are equivalent:
    \begin{itemize}
        \item[(1)] $\mu$ has exactly one atom, that is, the set $\{x\in X\colon \mu(\{x\})>0\}$ is a singleton,
        \item[(2)] there exists a $\nu\in\Wo(X)$, $\nu\neq\mu$ such that $M_{1/2}(\mu,\nu)$ is a singleton, but there are no $\eta,\vartheta\in\Wo(X)$, $\eta\neq\vartheta$ such that $M_{1/2}(\eta,\vartheta) = \{\mu\}$. 
    \end{itemize}
Indeed, this is straightforward by Lemma \ref{lem:0<p<1Dirac}. Next, using (1)$\iff$(2) we notice that the following are also equivalent:
\begin{itemize}
    \item[(i)] $\mu$ is a Dirac measure
    \item[(ii)] $\mu$ has exactly one atom and $\mathcal{N}(\mu)$ is dense in $\Wo(X)$.
\end{itemize}
 The direction (i)$\Longrightarrow$(ii) is obvious, since $\mathcal{N}(\delta_x)$ is plainly the set of all finitely supported measures. As for the (ii)$\Longrightarrow$(i) direction, write $\mu$ as $\mu = t\cdot\delta_x + \mu'$ where $\mu'$ has no atom. Clearly,
    \begin{align*}\mathcal{N}(\mu) = \left\{\sum_{j=1}^n \lambda_j\cdot\delta_{x_j} + \mu'\,\colon\, n\in\N, \lambda_j>0, \sum_{j=1}^n \lambda_j = t\right\}.
    \end{align*}
The closure of $\mathcal{N}(\mu)$ is the set $\{t\cdot\eta + \mu' \colon\, \eta\in\Wo(X)\}$ which coincides with $\Wo(X)$ if and only if $\mu=\delta_x$.

In light of the above we conclude that the image of any Dirac measure is again a Dirac measure, and thus the map $\psi$ defined by $\Phi(\delta_x)=:\delta_{\psi(x)}$ is an isometry of $X$. In fact, without loss of generality we may assume that $\Phi(\delta_x) = \delta_x$ for all $x$. 

What remains to be proven is that $\Phi$ fixes all finitely supported measures, which we shall do by using an induction. 
We already know this for measures with singleton support. Fix a $k\in\N$, $k\geq 1$ and suppose we proved the statement for measures supported on at most $k$ points. Take a measure $\mu$ supported on $k+1$ points. It is straightforward that $\mu$ can be expressed as $(1-\lambda)\cdot\mu_1+\lambda\cdot\mu_2$ with some $0<\lambda<1$ and $\mu_1,\mu_2\in\Wo(X)$ whose supports are sets with $k$ elements. We have
\begin{align*}
    \{\mu\} 
    = M_\lambda(\mu_1,\mu_2)
    = M_\lambda(\Phi(\mu_1),\Phi(\mu_2))
    = \Phi\ler{M_\lambda(\mu_1,\mu_2)}
    = \{\Phi\ler{\mu}\}.
\end{align*}
The proof is complete.
\end{proof}

Recall that the example given in Section \ref{s:example} shows that the above theorem is sharp in the sense that in general we cannot conclude isometric rigidity for $\Wp(X)$ if $1<p<\infty$.

Now, isometric rigidity of $\Wp(E)$ for $0<p<1$ is an immediate consequence of Theorem \ref{thm:concave-cost}. In fact, Theorem \ref{thm:concave-cost} implies the isometric rigidity of $\Wp(X)$ for $0<p<1$ and for all Polish space $(X,\varrho)$, as the $p$-th power of $\varrho$ satisfies the strict triangle inequality.

\begin{corollary}\label{cor:hilple1}
Let $(X,\varrho)$ be a complete separable metric space and $0<p<1$. Assume that $\Phi\colon\Wp(X)\to\Wp(X)$ is an isometry. Then there exists an isometry $\psi\in\isom(X)$ such that
\begin{equation*}
\Phi(\mu) = \psi_\#\mu \qquad (\mu\in\Wp(X)).
\end{equation*}
\end{corollary}
Recall that the proof of \eqref{eq:atom} works for the case $0<p<1$ as well, therefore in the Hilbert space case the above corollary could be also proved with the use of potential functions, once we know rigidity on Dirac masses.

Finally, we state another consequence of Theorem \ref{thm:concave-cost} about Wasserstein spaces built on ultrametric spaces. Various geometric properties of such spaces were described by Kloeckner in \cite{ultrametric}.

\begin{corollary}\label{cor:ultra}
Let $(X,\rho)$ be a complete, separable metric space. Suppose that $\rho$ is an ultrametric, that is,
$$
\rho(x,z) \leq \max\{\rho(x,y),\rho(y,z)\} \quad (x,y,z\in X).
$$
Let $0<p<\infty$ and $\Phi\colon\Wp(X)\to\Wp(X)$ be an isometry. Then there exists an isometry $\psi\in\isom(X)$ such that
\begin{equation*}
	\Phi(\mu) = \psi_\#\mu \qquad (\mu\in\Wp(X)).
\end{equation*}
\end{corollary}

\begin{proof} We only have to notice that 
    $\rho^p$ is a metric on $X$ which satisfies the strict triangle inequality, and
    that $\Wp(X,\rho)$ and $\Wo(X,\rho^p)$ contains exactly the same measures. Since $d_{\Wp(X,\rho)}^{\max\{1,p\}} \equiv d_{\Wo(X,\rho^p)}$, we can apply Theorem \ref{thm:concave-cost}.
\end{proof}

\section{Appendix}
\begin{proof}[Proof of Lemma \ref{lem:p>1Dirac}]
    For the direction (i)$\Longrightarrow$(ii) let $\mu = \delta_x$ with some $x\in E$, $\nu\in\Wp(E)$, $\nu\neq\mu$ and $T:=\dwp(\delta_x,\nu)$. By Lemma \ref{lem:geod} the curve 
	$$\gamma\colon [0,\infty)\to \Wp(E), \quad \gamma(t)=(D_x^{t/T})_\# \nu$$ 
	is a geodesic ray. Moreover, $\gamma|_{[0,T]}$ is the unique geodesic segment connecting $\mu$ with $\nu$, since there is only one coupling between them.

For the reverse direction, suppose that $\mu$ is not a Dirac measure but it satisfies (ii). Our aim is to obtain a contradiction. Fix an $x\in E$, set $T=\dwp(\mu,\delta_x)$ and consider the map
	$$
	\gamma\colon [0,T]\to \Wp(E), \quad \gamma(t)=(D_x^{1-t/T})_\# \mu.
	$$
	It is straightforward that $\gamma$ is the unique geodesic that connects $\gamma(0) = \mu$ and $\gamma(T) = \delta_x$. By our assumption, $\gamma$ extends to $[0,\infty)$, denote by $\tgamma$ this extension. Set $\nu = \tgamma(2T)$. Again, Lemma \ref{lem:geod} gives us that 
	$$
	\tgamma|_{[T,2T]}\colon [T,2T]\to \Wp(E),\quad \tgamma(t)=(D_x^{t/T-1})_\# \nu.
	$$
	Now, once again we apply Lemma \ref{lem:geod} to obtain a $\pi\in\Pi^0(\mu,\nu)$ which implements the geodesic segment $\tgamma|_{[0,2T]}$, that is, 
	$$
	\tgamma|_{[0,2T]}\colon [0,2T]\to \Wp(E), \quad \tgamma(t)=\left(g_{t/2T}\right)_\# \pi.
	$$
	In particular, $\delta_x = \tgamma(T) = \left(g_{1/2}\right)_\# \pi$,
	which implies that the support of $\pi$ is contained in $\{(2x-y,y)\colon y\in E\}$. Therefore, $\pi = (\id\times D_x^{-1})_\#\mu$, $\nu = (D_x^{-1})_\#\mu$.
	This means that the transport map $D_x^{-1}$ is optimal between $\mu$ and $\nu$. Observe that therefore the support of $\mu$ must be contained in a one dimensional affine subspace containing $x$. Indeed, otherwise it is easy to see that there exists a better transport plan, see Figure \ref{fig:better}. However, as $x$ was an arbitrary point, the same holds for all $x\in E$. Therefore, $\mu$ is concentrated on one point, a contradiction.
	\end{proof}

	\begin{figure}[H]
	\centering
	\includegraphics[scale=0.8]{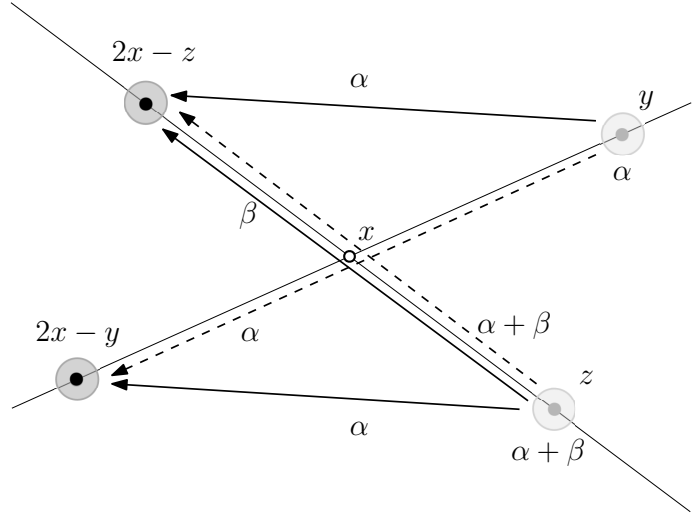}
	\caption{The grey points $y, z$ are in the support of $\mu$, the black ones are their images under the map $D_x^{-1}$, hence they are in the support of $\nu = (D_x^{-1})_\#\mu$. The white point $x$ is not contained in the line spanned by the grey points. The transport map $D_x^{-1}$ (dashed arrows) is not optimal between $\mu$ and $\nu$, since there is a better transport plan (black arrows). This latter plan is to be interpreted in the following way: we transport the mass in the ball around $y$ into the point $y$, then to the point $2x-z$, and finally to the ball around $2x-z$; the mass in the ball around $z$ is first transported into the point $z$, then some part is transported into $2x-y$ and the rest into $2x-z$, finally we transport the masses from these points to the balls around them. These steps are all done along straight line.}\label{fig:better}
\end{figure}
\bigskip

\begin{proof}[Proof of Lemma \ref{lem:p=1Dirac}]
Note that Lemma \ref{lem:geod} cannot be applied here directly. To prove (i)$\Longrightarrow$(ii) assume that $\mu=\delta_x$ for some $x\in E$. We claim that $\eta := (D_x^2)_\#\nu$ satisfies \eqref{eq:W1-Dirac}. 
	In order to see this, consider a sequence $\{\nu_k\}_{k=0}^\infty$ of finitely supported Borel probability measures such that $\lim_{k\to\infty}\dwo(\nu,\nu_k) = 0$. Then by Lemma \ref{lem:geod} we have
	$$
	\dwp(\delta_x,\nu_k) = \dwp\left(\nu_k,(D_x^2)_\#\nu_k\right) = \frac{1}{2}\dwp\left(\delta_x,(D_x^2)_\#\nu_k\right) \qquad (p>1, k\in\N).
	$$
	Set $\eta_k := (D_x^2)_\#\nu_k$ $(k\in\N)$. Then for all $1<p<\infty$, $k\in\N$ we obtain
\begin{align*}
	\int_{\supp(\nu_k)} \|x-y\|^p \dd\nu_k(y) &= \inf_{\pi \in \Pi(\nu_k, \eta_k)} \int_{\supp(\nu_k) \times \supp(\eta_k)} \|x-y\|^p~\dd \pi(x,y)\\
	&= \frac{1}{2^p}\int_{\supp(\eta_k)} \|x-y\|^p \dd\eta_k(y).
\end{align*}
Since both $\nu_k$ and $\eta_k$ are finitely supported, it is easy to see that as $p\to1+$ the above gives
	$$
	\dwo(\delta_x,\nu_k) = \dwo\left(\nu_k,(D_x^2)_\#\nu_k\right) = \frac{1}{2}\dwo\left(\delta_x,(D_x^2)_\#\nu_k\right) \qquad (k\in\N).
	$$
	Note that $\lim_{k\to\infty}\dwo\left((D_x^2)_\#\nu,(D_x^2)_\#\nu_k\right) = 2\lim_{k\to\infty}\dwo(\nu,\nu_k) = 0$. Therefore if we let $k\to\infty$, we obtain \eqref{eq:W1-Dirac}. Note also that if $T=\dwo(\delta_x,\nu)$, then a similar argument shows that $$\gamma\colon [0,\infty)\to \Wo(E), \quad \gamma(t)=(D_x^{t/T})_\# \nu$$ is a geodesic ray in $\Wo(E)$ for all $x\in E$ and $\nu\in\Wo(E)\setminus\{\delta_x\}$.
	
	To prove (ii)$\Longrightarrow$(i) suppose that $\mu$ is not a Dirac measure and that (ii) holds. Our aim is to get a contradiction from this. In such a case there are at least two different points, say $y,z\in E$ in the support of $\mu$. Consider another point $x\in E$ such that $x,y,z$ are not collinear, and set $\nu=\delta_x$. By our assumption, there exists an $\eta\in\Wo(E)$ such that \eqref{eq:W1-Dirac} holds. Since $p=1$, 
	the following transport plan is optimal between $\mu$ and $\eta$: transport everything first into $x$ along straight lines, then redistribute along straight lines to $\eta$. Obviously, this means that every straight line connecting any point of $\supp(\mu)$ and any point of $\supp(\eta)$ must contain $x$. 
	Therefore $\supp(\eta)\subseteq\{x\}$, a contradiction.
\end{proof}

Note that the above proof does not work in one dimension, however, the statement remains valid. Indeed, one can see this using quantile funcitons.
If $\mu$ is not a Dirac measure, then there exists a $t\in\R$ such that $\inf F_{\mu}^{-1}<t\equiv F_{\delta_t}^{-1}<\sup F_{\mu}^{-1}$, and $\mu$ cannot be reflected through $\delta_t$ in the sense of \eqref{eq:W1-Dirac}.\\
\bigskip
\begin{proof}[Proof of Lemma \ref{lem:transl}]
	For any $\pi\in\Pi(\mu,\nu)$ we have $(t_{(v,0)})_\#\pi\in\Pi((t_v)_\#\mu,\nu)$, and vica--versa. Hence
	\begin{align*}
	\dwt^2\left((t_v)_\#\mu,\nu\right) &= \inf_{\pi \in \Pi(\mu, \nu)} \int_{E \times E} \norm{x-y}^2~\dd \ler{(t_{(v,0)})_\#\pi}(x,y) \\
	&= \inf_{\pi \in \Pi(\mu, \nu)} \int_{E \times E} \norm{x+v-y}^2~\dd \pi(x,y)\\
	&= \inf_{\pi \in \Pi(\mu, \nu)} \int_{E \times E} \norm{x-y}^2+\norm{v}^2+2\inner{x}{v}-2\inner{y}{v}~\dd \pi(x,y) \\
	&= \dwt^2\left(\mu,\nu\right) + \norm{v}^2 + 2\int_{E} \inner{x}{v}~\dd \mu(x) -2\int_{E}\inner{y}{v}~\dd \nu(y),
	\end{align*}
	which gives \eqref{eq:transl}. The identity \eqref{eq:bari-transl} follows if we translate both arguments in the left-hand side by the vector $m(\nu)$.
\end{proof}
\bigskip

\begin{proof}[Proof of Lemma \ref{lem:ort}]
	Consider $\widehat{\pi} := \mu\otimes\nu$. Then
	\begin{align*}
	\int_{E\times E} \norm{x-y}^2 \dd\widehat{\pi}(x,y) &= \int_{E\times E} \norm{(x-m(\mu)) - (y-m(\nu)) + (m(\mu)-m(\nu))}^2 \dd\widehat{\pi}(x,y) \\
	&= \sigma^2 + \rho^2 + \norm{m(\mu)-m(\nu)}^2,
	\end{align*}
	where we used the Cauchy--Schwartz inequality and Fubini's theorem in order to see that $\int_{E\times E} \inner{x}{y} \dd\widehat{\pi}(x,y) = \langle m(\mu),m(\nu)\rangle$.
	Therefore, \eqref{eq:ort} holds if and only if the coupling $\widehat{\pi}$ is optimal.
	
	If $\mu$ and $\nu$ are supported on orthogonal affine subspaces, then every coupling is optimal by the Pythagorean theorem.
	On the other hand, if they are not supported on orthogonal affine subspaces, then there exist points $x,y\in \supp(\mu)$ and $z,t\in \supp(\nu)$ such that $\inner{x-y}{z-t} < 0$. By a short calculation we obtain $$\norm{x-z}^2+\norm{y-t}^2 > \norm{x-t}^2+\norm{y-z}^2.$$ In particular this means, that it is better to transport $\varepsilon$ mass from $x$ to $t$ and $\varepsilon$ mass from $y$ to $z$, then to transport $\varepsilon$ mass from $x$ to $z$ and $\varepsilon$ mass from $y$ to $t$. Of course this property also holds for points close enough to $x,y,z,t$. Therefore, we see that either $(x,z)$ or $(y,t)$ cannot be in the support of an optimal coupling. So, $\mu\otimes\nu$ is not optimal. 
\end{proof}
\bigskip

\begin{proof}[Proof of Lemma \ref{lem:bisector}]
	Note that any $\pi\in\Pi(\zeta_{a,b}^{\alpha}(x),\mu)$ can be written in the form $$\pi = \delta_{ax}\otimes\mu^a_\alpha + \delta_{bx}\otimes\mu^b_\alpha$$ with some positive measures $\mu^a_\alpha, \mu^b_\alpha$, $\mu^a_\alpha+\mu^b_\alpha = \mu$, $\mu^a_\alpha(E)=\alpha$, $\mu^b_\alpha(E)=1-\alpha$. Denote by $S_{ax}(ax,bx)$ and $S_{bx}(ax,bx)$ the open halfspaces containing $ax$ and $bx$, respectively, and whose boundaries are the bisector $B(ax,bx)$. Assume that $\mu\ler{S_{ax}(ax,bx)}\leq\alpha$ and $\mu\ler{S_{bx}(ax,bx)}\leq1-\alpha$. Then it is obvious that any coupling $\delta_{ax}\otimes\mu^a_\alpha + \delta_{bx}\otimes\mu^b_\alpha$ such that 
	\begin{align*}\supp\ler{\mu^a_\alpha}\subset S_{ax}(ax,bx)\cup B(ax,bx),\quad \supp(\mu^b_\alpha)\subset S_{bx}(ax,bx)\cup B(ax,bx)
	\end{align*}
	is optimal. In particular, we have 
	\begin{align*}
	\dwp^p(\mu,\zeta_{a,b}^{\alpha}(x)) = \int_{E} \min\{\norm{ax-y}^p, \norm{bx-y}^p\}\dd\mu(y).
	\end{align*}
	Now suppose that $\mu\ler{S_{ax}(ax,bx)}>\alpha$ (the case when $\mu\ler{S_{bx}(ax,bx)}>1-\alpha$ is similar). Choose an optimal coupling $\pi = \delta_{ax}\otimes\mu^a_\alpha + \delta_{bx}\otimes\mu^b_\alpha$. Then $\mu^b_\alpha$ cannot be supported on $S_{bx}(ax,bx)\cup B(ax,bx)$, since $\mu\ler{S_{bx}(ax,bx)\cup B(ax,bx)} < 1-\alpha = \mu^b_\alpha(E)$. For any $\delta>0$ set 
$$W_\delta := \{y\in E\colon \norm{ax-y}^p+\delta < \norm{bx-y}^p\}.$$
It is straightforward that there exists a $\delta>0$ such that $\mu^b_\alpha(W_\delta) > 0$. Hence,
	\begin{align*}
	\dwp^p(\mu,\zeta_{a,b}^{\alpha}(x)) &= \int_E \norm{ax-y}^p\dd\mu^a_\alpha(y) + \int_E \norm{bx-y}^p\dd\mu^b_\alpha(y)\\
	&> \int_{E\setminus W_\delta} \min\{\norm{ax-y}^p, \norm{bx-y}^p\}\dd\mu{(y)}+ \\
	&\hspace{2cm}+ \int_{W_\delta} \norm{ax-y}^p\dd\mu^a_\alpha{(y)} + \int_{W_\delta} \ler{\norm{ax-y}^p+\delta}\dd\mu^b_\alpha{(y)} \\
	&> \int_{E} \min\{\norm{ax-y}^p, \norm{bx-y}^p\}\dd\mu{(y)} + \mu^b_\alpha\ler{W_\delta} \delta.
	\end{align*}
	
	Therefore $m = \int_{E} \min\{\norm{ax-y}^p, \norm{bx-y}^p\}\dd\mu(y)$, and the proof is done.
\end{proof}
\bigskip

\bigskip
\begin{proof}[Proof of Theorem \ref{thm:marad}]
The existence of optimal transport plans is a consequence of the tightness of $\mu$ and $\nu$, which is guaranteed by $X$ being a Polish space, see \cite[pp. 133 and pp. 108]{Ambrosio}.
From here we prove our statement along the lines of \cite[Theorem 2.2]{Pegon}. It is enough to prove that $(\proj_1)_\#(\pi|_{X\setminus D})$ and $(\proj_2)_\#(\pi|_{X\setminus D})$ are singular to each other, for which it suffices to show that 
$$\proj_1[\supp(\pi)\setminus D]\cap\proj_2[\supp(\pi)\setminus D]=\emptyset$$ 
Assume this is not the case, then there exist $x,y,z\in X$ such that $(z,x), (y,z)\in \supp(\pi)\setminus D$. However, since we have
\begin{align*}
    \rho(z,x)+\rho(y,z) > \rho(x,y) + \rho(z,z),
\end{align*}
$\supp(\pi)$ is not $c$--monotone (\cite[Definition 6.1.3]{Ambrosio}), which by \cite[Theorem 6.1.4]{Ambrosio} contradicts the optimality of $\pi$.
\end{proof}

\section{Acknowledgements}
This paper is based on discussions made during research visits at the Institute of Science and Technology (IST) Austria, Klosterneuburg. We are grateful to the Erd\H{o}s group for the warm hospitality.
We are also grateful to Lajos Moln\'ar for his comments on an earlier version of the manuscript and to L\'aszl\'o Erd\H{o}s for his suggestions on the structure and highlights of this paper.

\end{document}